\documentclass[reqno,english]{amsart}
\usepackage{amsfonts,amsmath,latexsym,verbatim,amscd,mathrsfs,color,array}

\usepackage{amsmath,amssymb,amsthm,amsfonts,graphicx,color}
\usepackage{amssymb}
\usepackage{pdfsync}
\usepackage{epstopdf}
\usepackage[colorlinks=true]{hyperref}
\usepackage{subfigure}
\usepackage{extarrows}

\makeatletter
\def\@currentlabel{2.1}\label{e:dispaa}
\def\@currentlabel{2.21}\label{e:dispau}
\def\@currentlabel{2.22}\label{e:dispav}
\def\@currentlabel{2.23}\label{e:dispaw}
\def\@currentlabel{2.24}\label{e:dispax}
\def\theequation{\thesection.\@arabic\c@equation}
\makeatother
\let\oldbibliography\thebibliography
\renewcommand{\thebibliography}[1]{%
\oldbibliography{#1}%
\setlength{\itemsep}{0pt}%
}

\renewcommand{\theequation}{\thesection.\arabic{equation}}
\newtheorem{lemma}{Lemma}[section]

\newtheorem{definition}{Definition}

\newtheorem{proposition}{Proposition}[section]
\newtheorem{corollary}{Corollary}[section]
\newtheorem{remark}{Remark}[section]
\newtheorem{conjecture}{Conjecture}[section]
\newtheorem{open problem}{Open Problem}[section]
\newtheorem{open question}{Open Quesion}[section]
\newcommand{\bremark}{\begin{remark} \em}
\newcommand{\eremark}{\end{remark} }

\newtheorem{numerical/experimental results}{Numerical/Experimental results}[section]

\newtheorem{theorem}{Theorem}[section]
\newtheorem*{thmA}{Theorem A}

\newcommand{\BE}{\begin{equation}}
\newcommand{\BEN}{\begin{equation*}}
\newcommand{\EE}{\end{equation}}
\newcommand{\EEN}{\end{equation*}}
\newcommand{\BL}{\begin{lemma}}
\newcommand{\EL}{\end{lemma}}
\newcommand{\BT}{\begin{theorem}}
\newcommand{\ET}{\end{theorem}}
\newcommand{\BP}{\begin{proposition}}
\newcommand{\EP}{\end{proposition}}
\newcommand{\BC}{\begin{corollary}}
\newcommand{\EC}{\end{corollary}}
\renewcommand{\Re}{\operatorname{Re}}
\renewcommand{\Im}{\operatorname{Im}}

\begin{document}


\title[Lattice energy and hexagonal minimizer]{\bf Minimizing lattice Energy and hexagonal crystallization}

\author{Kaixin Deng}
\author{Senping Luo}
\address[K.~Deng]{School of Mathematics and statistics, Jiangxi Normal University, Nanchang, 330022, China}
\address[S.~Luo]{School of Mathematics and statistics, Jiangxi Normal University, Nanchang, 330022, China}

\email[K.~Deng]{Dengkaikai1999@126.com}
\email[S.~Luo]{luosp14@tsinghua.org.cn}

\begin{abstract}
Consider the energy per particle on the lattice
\begin{equation}\aligned\nonumber
\min_{  \Lambda }\sum_{ \mathbb{P}\in  \Lambda} \left|\mathbb{P}\right|^4 e^{-\pi \alpha \left|\mathbb{P}\right|^2 }
\endaligned\end{equation}
where $\alpha >0$ and $\Lambda$ is a two dimensional lattice. We prove that for $\alpha\geq\frac{3}{2}$, among two dimensional lattices with unit density, such energy minimum is attained at $e^{i\frac{\pi}{3}}$, corresponding to hexagonal lattice. Our result partially answers some questions raised in
 \cite{Bet2019AMP,Luo2023}.

\end{abstract}

\maketitle


\section{Introduction and main results}
\setcounter{equation}{0}

\subsection{Minimization of energy on lattices}
In this paper, we consider energy minimization and the hexagonal minimizer. The problem of energy minimization has deep
applications and connections in physics. In authors' perspective, energy minimization has an inner connection to the following problem, which is a follows up a 1984 list of open problems in mathematical physics proposed by Simon in \cite{Simon2000}(2000).

\begin{open problem}[Simon 1984 and 2000 \cite{Simon1984,Simon2000}]\label{ProbA}
Is there a mathematical sense in which one can justify from first principles current techniques for determining molecular configurations?
\end{open problem}
Simon's problem \ref{ProbA} is very vague. In our perspective, by the first principle, one considers the Hamiltonian of the admissible system, among all possible admissible configurations, it is aimed to show that the minimum of the Hamiltonian is attained at targeted molecular configuration. We state it in a mathematical way
 \begin{equation}\aligned\label{EA}
Minima_{\mathcal{A}\subset\hbox{admissible configurations}} E(\mathcal{A})=\hbox{targeted molecular configuration},
\endaligned\end{equation}
where $E(\mathcal{A})$ denotes the Hamiltonian of the system.

The Hamiltonian of our problem is given by the two-body potential:
 \begin{equation}\aligned\nonumber
H=\sum_{1\leq i<j\leq N}\phi(|x_i-x_j|^2),
\endaligned\end{equation}
where the particles move in a volume of $V$ of a d-dimensional space, and $\phi$ is the potential of the system.
We then take the thermodynamic limit $N, V\rightarrow\infty$ of the Hamiltonian at fixed density $\rho=\frac{N}{V}$, namely
 \begin{equation}\aligned\label{HN}
\lim_{N\rightarrow\infty, \;\rho=\frac{N}{V}\;\; \hbox{fixed}}\frac{H}{N}.
\endaligned\end{equation}

In this paper, we restrict the admissible configurations to the lattices. This is physically and chemically meaningful since there are lots of crystal structures in nature. The crystal structures are important objects in physics and chemistry, while in mathematics, the most suitable concept to crystal structures is lattices. Note that in the lattice system, the thermodynamic limit of the Hamiltonian given by \eqref{HN} is reduced to
 \begin{equation}\aligned\label{HN2}\sum_{{P}\in \Lambda \backslash\{0\}} \phi(|{P}|^2).
 \endaligned\end{equation}

A derivation from \eqref{HN} to \eqref{HN2} on lattices can be found in Luo-Wei-Zou (\cite{LW2021}). By this way, the Simon's problem \ref{ProbA} is precisely formulated by

 \begin{equation}\aligned\label{EFL}
\min_{\Lambda} E_{\phi}(\Lambda ), \;\;\hbox{where}\;\;E_{\phi}(\Lambda):=\sum_{{P}\in \Lambda \backslash\{0\}} \phi(|{P}|^2).
\endaligned\end{equation}
Here $\Lambda$ is a d-dimensional lattice, and $\phi$ is the two-body potential of the system. We shall focus on dimension two, since {\it the two dimensional theories capture many essential features of higher dimensions, without sharing the complexities of higher dimensions}(Alvarez-Gaum\'e-Moore-Vafa \cite{Vafa1986}).

The hexagonal shape is almost ubiquitous in nature, consequently a particular and intrigue interest is when the molecular configuration (Simon's problem \ref{ProbA}) is hexagonal shape (an illustration of hexagonal shape in Figure \ref{h12}). Gruber summarized in \cite{Gruber2000} and further we ask that
\begin{open problem}[Gruber 2000 \cite{Gruber2000}]\label{ProbB}
 Why in many cases, optimal configurations are
 almost regular hexagonal?
\end{open problem}

\begin{figure}
\centering
 \includegraphics[scale=0.4]{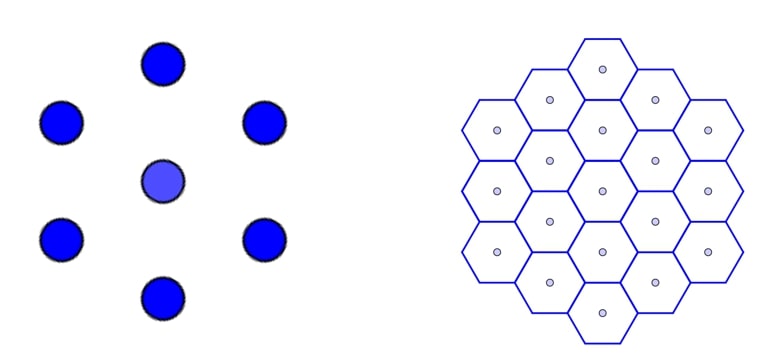}
 \caption{The hexagonal structure and periodic hexagonal structure.}
 \label{h12}
\end{figure}

In the framework of the lattices minimization problem \eqref{EFL}, Open Problem \ref{ProbB} becomes a problem asked in Luo-Wei \cite{Luo2023}.

\begin{open problem}[Luo-Wei \cite{Luo2023}]\label{Open1}
Finding the non-monotone potentials, such that the minimum  of $\Lambda\rightarrow E_{\phi}( \Lambda)$ is achieved at hexagonal lattice?
\end{open problem}

A one-step further problem to Open Problem \ref{Open1} is

\begin{open problem}[B\'etermin-Petrache \cite{Bet2019AMP},
Luo-Wei \cite{Luo2023}]\label{Open2}
 Finding the largest possible class of potentials such that $\Lambda\rightarrow E_{\phi}( \Lambda)$ is achieved at hexagonal lattice?
\end{open problem}
Open Problem \ref{Open2} is related to the universality optimality (i.e., global minimum among all periodic configuration at fixed density) in 2d; it is solved recently in dimensions 8 and 24 by Cohn et al. \cite{Cohn2022}, while it is still open in dimension 2. A related Conjecture about hexagonal minimizer on some general configurations was made by Sandier-Serfaty \cite{Serfaty2012}.
\vskip0.1in
Independently, the lattice energy
$E_{\phi}(\Lambda )$ given by \eqref{EFL} denotes the energy per particle of the system under the background potential $\phi$ over a periodical lattice $\Lambda$, and the function ${\phi}$ denotes the potential of the system and the summation ranges over all the lattice points except for the origin $0$. The lattices minimization problem \eqref{EFL} roots in many physical problems (B\'etermin \cite{Bet2016,BP,Bet2020}), chemical problems, and analytical number theory (\cite{Bet2015,Bet2016,Bet2018,Bet2019,Bet2019AMP,Betermin2021JPA,Luo2019, Serfaty2018}).

\subsection{Classical results in analytical number theory}

 Recall the Theta function defined as

 \begin{equation}\aligned
 \nonumber
\theta (\alpha; \Lambda):=\sum_{\mathbb{P}\in\Lambda} e^{- \pi\alpha |\mathbb{P}|^2}.
\endaligned\end{equation}

Let $ z\in \mathbb{H}:=\{z\in\mathbb{C}: \Im(z)>0\}$ and $\Lambda =\frac{1}{\sqrt{\Im(z)}}\big({\mathbb Z}\oplus z{\mathbb Z}\big)$ be the lattices with unit density in $ \mathbb{R}^2$. One has the explicit form of Theta function:
 \begin{equation}\aligned
 \label{thetas}
\theta (\alpha; z)=\sum_{(m,n)\in\mathbb{Z}^2} e^{-\pi\alpha \frac{ |mz+n|^2}{\Im(z) }}.
\endaligned\end{equation}

The Theta function given by \eqref{thetas} has strong connections to the energy given by \eqref{EFL}, in fact, when the potential is Gaussian form ($\phi(r^2)=e^{- \pi\alpha {r^2}}, \alpha>0$), the corresponding $E_{\phi} (\Lambda)$ becomes the Theta function.

In 1988, Montgomery \cite{Mon1988} proved the following celebrated result:

\begin{thmA}[Montgomery 1988] For all $\alpha>0$, up to rotations and translations,
\begin{equation}\aligned\nonumber
\underset{z\in\mathbb{H}}{\mathrm{Minima}}\:\theta(\alpha;z)=e^{i\frac{\pi}{3}}\:\:( \mathrm{corresponding \:\:to\:\:haxagonal\:\: lattice}).
\endaligned\end{equation}
\end{thmA}
Montgomery's Theorem A is of great significance for the fields of physic, chemistry and has a series of implications in Gaussian Gabor systems \cite{Thomas 2003}, Bose-Einstein condensates \cite{Luo2022}, quantum mechanics (\cite{Markus2020}, \cite{Marku2018}), superconductors in large magnetic field \cite{Abr}, coding and data transmission (chapter 33 of geometry of numbers \cite{Discrete}), packing of balls (chapter 29 of geometry of numbers \cite{Discrete}), materials science (\cite{Joshua 2019,Geim 2007}), nanomedicine \cite{Albarqi2019}, although he investigated it purely out of the interest from number theory at the time. The significant result also laid the foundations for a host of optimal lattice problems.

Crystallization is called one of the most mysterious processes in nature ( Uzunova-Pan-Lubchenko-Vekilov \cite{Uzunova 2012}). {\it The mystery lies in how structures with long-range order form from building blocks that only interact with their local neighbors} (Lutsko \cite{Lutsko2019}). Therefore, it is a fundamental problem to understand the formation of crystals at cryogenic temperatures and interpret the interactions of particles  and  other conditions  how to affect crystal  structure.

Motivated by Luo-Wei \cite{LW2022,Luo2023}, the background potential that we consider is ${\phi}_{1}(r^2)=r^{4} e^{- \pi\alpha {r^2}} $. Our main result is stated as follows

\begin{theorem}\label{Th2} Assume $\alpha\geq \frac{3}{2}$. Let $\mathcal{R}(\alpha;\Lambda)$ be defined as
\begin{equation}\aligned\label{R}
\mathcal{R}(\alpha;\Lambda):=\sum_{\mathbb{P}\in \Lambda} {\left|\mathbb{ P} \right|}^{4}e^{-\pi\alpha {\left| \mathbb{P} \right|}^{2}},
\endaligned\end{equation}
where $\Lambda =\frac{1}{\sqrt{\Im(z)}}\big({\mathbb Z}\oplus z{\mathbb Z}\big),z\in\mathbb{H}$. Then among two dimensional lattices with unit density,
the minimizer of $\mathcal{R}(\alpha;\Lambda)$ always exists and is always achieved at hexagonal lattice.
\end{theorem}

The functional $\mathcal{R}(\alpha;\Lambda)$ has an elegant form and many consequences. We state one of them as follows

\begin{corollary}\label{cor1}
Assume that $\beta> \alpha\geq \frac{3}{2}$. Consider the minimizing problem
\begin{equation}\aligned\nonumber
\underset{\Lambda}{\min}\sum_{\mathbb{P}\in \Lambda} {\left| \mathbb{P} \right|}^{2}\Big(e^{-\pi\alpha {\left| \mathbb{P} \right|}^{2}}  -e^{-\pi\beta {\left| \mathbb{P} \right|}^{2}} \Big),
\endaligned\end{equation}
where $\Lambda =\frac{1}{\sqrt{\Im(z)}}\big({\mathbb Z}\oplus z{\mathbb Z}\big),z\in\mathbb{H}$. Then among two dimensional lattices with unit density, the minimizer of the lattice energy functional always exists and is always achieved at hexagonal lattice.
\end{corollary}
Note that the specific functional in Corollary \ref{cor1} corresponds to the functional given by \eqref{EFL} with the potential $\phi_{2}(r^2)=r^2e^{-\alpha\pi{r^2}}-r^2e^{-\beta\pi{r^2}}$. When potential functions $\hat{\phi}_{1}(r^2)=r^4{e}^{-\pi{r^2}}, \hat{\phi}_{2}(r^2)$ $=r^2e^{-\pi{r^2}}-r^2e^{-2\pi{r^2}}$, their images are shown in Figure \ref{p}. Notice that they are both non-monotone, therefore, Theorem \ref{Th2} and Corollary \ref{cor1} partially answer Problems \ref{Open1}-\ref{Open2}.
To fully understand Problems \ref{Open1}-\ref{Open2}, we propose that
\begin{figure}
\centering
 \includegraphics[scale=0.38]{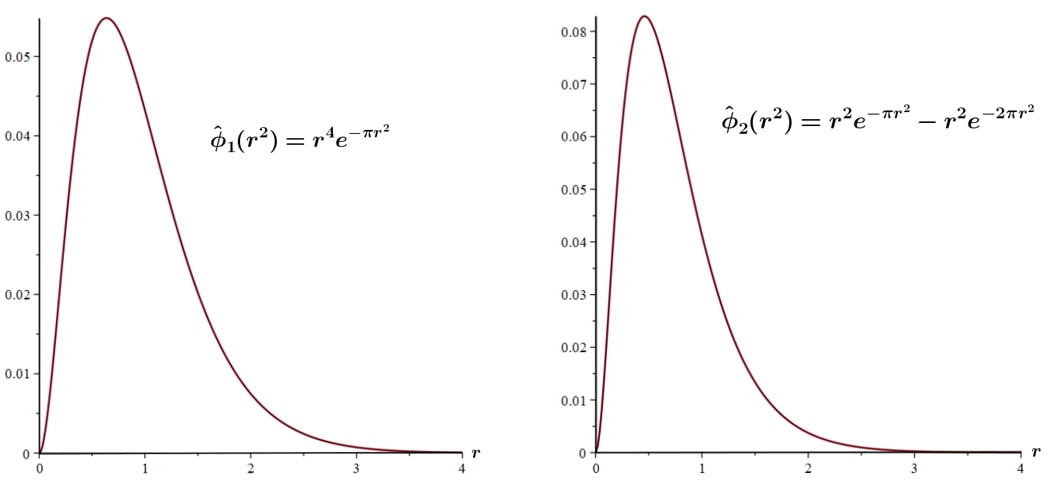}
 \caption{The images of $\hat{\phi}_{1}(r^2)$ and $\hat{\phi}_{2}(r^2)$.}
 \label{p}
\end{figure}

\begin{conjecture}
Assume that $k\in {\mathbb{Z}}^{+}, \alpha\geq k$, then among two dimensional lattices with unit density,
\begin{equation}\aligned\nonumber
\underset{\Lambda}{\min}\sum_{\mathbb{P} \in \Lambda} {\left| \mathbb{P} \right|}^{2k}e^{-\pi\alpha {\left| \mathbb{P} \right|}^{2}}
\endaligned\end{equation}
always exists and is always achieved at hexagonal lattice.
\end{conjecture}

The following parts of the paper is organized as follows: in Section 2, some useful properties of the functionals are  characterized and the estimates of the related forms of the Jacobi theta functions are provided. In Section 3, we prove that the minimum of the functional can be achieved on the right boundary of the fundamental domain. In Section 4, we prove that the minimization on the right boundary of the fundamental domain can be reduced to the hexagonal point (see Figure \ref{h}). Finally, Section 5 contains the proofs of Theorem \ref{Th2} and Corollary \ref{cor1}.

\section{Preliminaries}
\setcounter{equation}{0}

   The topic under discussion is closely related to the Theta function given by \eqref{thetas}. Through using the symmetries of the Theta function and the relationship between the function $\mathcal{R}$ given by \eqref{R} and the Theta function, it reduces the location of the minimum from the upper half plane to the fundamental domain. In particular, the analysis of this problem relies heavily on the estimates of the quotients of derivatives of the Jacobi theta functions. Due to the complexity of this problem, we have to refine some estimates of the quotients of derivatives of the Jacobi theta functions. The estimates can be used not only in this paper, but also in the related problems.

 We use the following definition of fundamental domain which is slightly different from the classical definition (see \cite{Mon1988}):
\begin{definition}[page 108, \cite{Eva1973}]
The fundamental domain associated to group $G$ is a connected domain $\mathcal{D}$ that satisfies
\begin{itemize}
  \item For any $z\in\mathbb{H}$, there exists an element $\pi\in G$ such that $\pi(z)\in\overline{\mathcal{D}}$;
  \item Suppose $z_1,z_2\in\mathcal{D}$ and $\pi(z_1)=z_2$ for some $\pi\in G$, then $z_1=z_2$ and $\pi=\pm Id$.
\end{itemize}
\end{definition}

Let
$
\mathbb{H}
$
 denote the upper half plane and  $\mathcal{S} $ denote the modular group
\begin{equation}\aligned\label{modular}\nonumber
\mathcal{S}:=SL_2(\mathbb{Z})=\{
\left(
  \begin{array}{cc}
    a & b \\
    c & d \\
  \end{array}
\right), ad-bc=1, a, b, c, d\in\mathbb{Z}
\}.
\endaligned\end{equation}.

By Definition 1, the fundamental domain associated to modular group $\mathcal{S}$ is
\begin{equation}\aligned\label{Fd1}\nonumber
\mathcal{D}_{\mathcal{S}}:=\{
z\in\mathbb{H}: |z|>1,\; -\frac{1}{2}<x<\frac{1}{2}
\}.
\endaligned\end{equation}
Note that the fundamental domain can be open (See [page 30, \cite{Apo1976}]).

Next we introduce another group related  to the functional $\theta(\alpha;z)$. The generators of the group are given by
\begin{equation}\aligned\label{GroupG1}
\mathcal{G}: \hbox{the group generated by} \;\;\tau\mapsto -\frac{1}{\tau},\;\; \tau\mapsto \tau+1,\;\;\tau\mapsto -\overline{\tau}.
\endaligned\end{equation}

It is easy to see that
the fundamental domain associated to group $\mathcal{G}$ denoted by $\mathcal{D}_{\mathcal{G}}$ is
\begin{equation}\aligned\label{Fd3}
\mathcal{D}_{\mathcal{G}}:=\{
z\in\mathbb{H}: |z|>1,\; 0<x<\frac{1}{2}
\}.
\endaligned\end{equation}

\begin{figure}
\centering
 \includegraphics[scale=0.5]{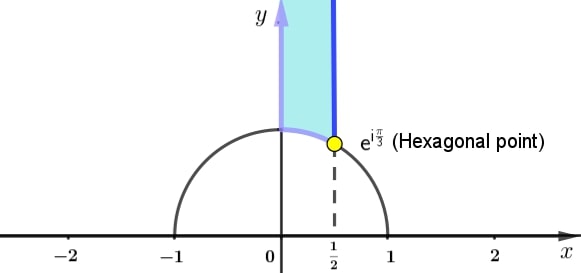}
 \caption{The hexagonal point in the fundamental domain.}
 \label{h}
\end{figure}

The following Lemma provided by Luo-Wei \cite{Luo2022} characterizes the fundamental symmetries of the Theta function $\theta (\alpha; z)$.
\begin{lemma}[Luo-Wei \cite{Luo2022}]\label{G111} For any $\alpha>0$, any $\gamma\in \mathcal{G}$ and $z\in\mathbb{H}$,
$\theta (\alpha; \gamma(z))=\theta (\alpha; z)$.
\end{lemma}

Based on the definition of the function $\mathcal{R}$ given by \eqref{R}, one has a more explicit expression
\begin{equation}\aligned\label{RR}
\mathcal{R}(\alpha;z)=\sum_{ (m,n)\in \mathbb{Z} ^2 }\frac{{\left| mz+n \right|}^4}{{{\Im}^2(z)}}e^{-\pi\alpha\frac{ \left|mz+n\right|^2}{\Im(z)}}.
\endaligned\end{equation}

A direct checking by \eqref{thetas} and \eqref{RR} shows that
\begin{lemma}[The relationship between the function $\mathcal{R}$ and the Theta function]\label{respect}
\begin{equation}\aligned\label{respectt}\nonumber
\mathcal{R}(\alpha;z)=\frac{1}{\pi^2}\frac{\partial^2}{\partial {\alpha}^2} \theta(\alpha;z).
\endaligned\end{equation}
\end{lemma}

Combining Lemmas \ref{G111} and \ref{respect}, we have the following invariance for $\mathcal{R} (\alpha;z)$.
\begin{lemma}[The invariance under group action for $\mathcal{R} (\alpha;z)$]\label{INR}
For any $\alpha>0$, any $\gamma\in \mathcal{G}$ and $z\in\mathbb{H}$, it holds that
\begin{equation}\aligned\label{INRR}\nonumber
\mathcal{R} (\alpha; \gamma(z))=\mathcal{R} (\alpha; z).
\endaligned\end{equation}
\end{lemma}

We also need some delicate analysis of the Jacobi theta function. We first recall the following well-known Jacobi triple product formula:
\begin{equation}\aligned\label{Jacob1}
\prod_{m=1}^\infty(1-x^{2m})(1+x^{2m-1}y^2)(1+\frac{x^{2m-1}}{y^2})=\sum_{n=-\infty}^\infty x^{n^2} y^{2n},
 \endaligned\end{equation}
for complex numbers $x,y$ with $|x|<1$, $y\neq0$.

The Jacob theta function is defined as
\begin{equation}\aligned\label{Jacobi}\nonumber
\vartheta_J(z;\tau):=\sum_{n=-\infty}^\infty e^{i\pi n^2 \tau+2\pi i n z},
 \endaligned\end{equation}
and the classical one-dimensional theta function is given by
\begin{equation}\aligned\label{TXY}
\vartheta(X;Y):=\vartheta_J(Y;iX)=\sum_{n=-\infty}^\infty e^{-\pi n^2 X} e^{2n\pi i Y}.
 \endaligned\end{equation}
By the Poisson summation  formula, it holds that
\begin{equation}\aligned\label{Poisson}
\vartheta(X;Y)=X^{-\frac{1}{2}}\sum_{n=-\infty}^\infty e^{-\pi \frac{(n-Y)^2}{X}} .
 \endaligned\end{equation}
Hence by the Jacobi triple product formula \eqref{Jacob1}, it holds that
\begin{equation}\aligned\label{Product}
\vartheta(X;Y)=\prod_{n=1}^\infty(1-e^{-2\pi n X})(1+e^{-2(2n-1)\pi X}+2e^{-(2n-1)\pi X}\cos(2\pi Y)).
 \endaligned\end{equation}

\subsection{Some useful estimates}
To estimate bounds of quotients of derivatives of $\vartheta(X;Y)$, we denote that
\begin{equation}\aligned\label{duct}
\mu(X)&=\sum_{n=2}^{\infty}n^2e^{-\pi (n^2-1) X},\:\:\:\:\:\:\:\:\:\:\:\:\hat{\mu}(X)=\sum_{n=2}^{\infty}n^2e^{-\pi (n^2-1) X}(-1)^{n+1},\\
\nu(X)&=\sum_{n=2}^{\infty}n^4e^{-\pi (n^2-1) X},\:\:\:\:\:\:\:\:\:\:\:\:\hat{\nu}(X)=\sum_{n=2}^{\infty}n^4e^{-\pi (n^2-1) X}(-1)^{n+1},\\
\omega(X)&=\sum_{n=2}^{\infty}n^6e^{-\pi (n^2-1) X},\:\:\:\:\:\:\:\:\:\:\:\hat{\omega}(X)=\sum_{n=2}^{\infty}n^6e^{-\pi (n^2-1) X}(-1)^{n+1}.
 \endaligned\end{equation}.

The following Lemmas \ref{Lemmamoist}-\ref{Lemmamoistt} proved by Luo-Wei \cite{Luo2023} are helpful in our subsequent analysis.
\begin{lemma}[Luo-Wei \cite{Luo2023}]\label{Lemmamoist}
Assume that $Y>0,k\in \mathbb{N^{+}}$. It holds that
 \begin{itemize}
  \item [(1)] $\left|\frac{\vartheta_{Y}(X;kY)}{\vartheta_{Y}(X;Y)} \right|\leq k\cdot \frac{1+\mu(X)}{1-\mu(X)}$ for $X> \frac{1}{5};$
  \item [(2)] $\left|\frac{\vartheta_{Y}(X;kY)}{\vartheta_{Y}(X;Y)} \right|\leq k\cdot \frac{1}{\pi}e^{\frac{\pi}{4X}}$ for $X< \frac{\pi}{\pi+2}$.
\end{itemize}
\end{lemma}


\begin{lemma}[Luo-Wei \cite{Luo2023}]\label{Lemmamoistt}
Assume that  $Y>0,k\in \mathbb{N^{+}}$. It holds that
 \begin{itemize}
   \item [(1)] $\left|\frac{\vartheta_{XY}(X;kY)}{\vartheta_{Y}(X;Y)} \right|\leq k\pi \cdot \frac{1+\nu(X)}{1-\mu(X)}$ for $X\geq \frac{1}{5};$
  \item  [(2)]for $k=1$,we have the more precise bound $\left|\frac{\vartheta_{XY}(X;kY)}{\vartheta_{Y}(X;Y)} \right|\leq  \pi \cdot  \frac{1+\nu(X)}{1+\mu(X)}$ for $X\geq \frac{1}{5};$
  \item  [(3)] $\left|\frac{\vartheta_{XY}(X;kY)}{\vartheta_{Y}(X;Y)} \right|\leq  \frac{3k}{2\pi}X^{-1}(1+\frac{\pi}{6}\frac{1}{X})e^{\frac{\pi}{4X}}$ for $X\leq \frac{1}{2};$
   \item [(4)] $\left|\frac{\vartheta_{XY}(X;Y)}{\vartheta_{Y}(X;Y)} \right|\leq \frac{3}{2}X^{-1}(1+\frac{\pi}{6}\frac{1}{X})$ for $X\leq \frac{1}{2}.$
 \end{itemize}
\end{lemma}

The estimates of $\frac{\vartheta_{XY}(X;Y)}{\vartheta_{Y}(X;Y)}$ in the Lemma \ref{Lemmamoistt} need to be further refined and will be presented in the following Lemma. The proof of Lemma \ref{Lemmatime} need two auxiliary Lemmas \ref{mono} and \ref{Lemmaair1}, which are placed in the Subsection 2.2.

\begin{lemma}\label{Lemmatime}
Assume that $Y\in\mathbb{R}$, we have the more precise upper and lower bounds for $\frac{\vartheta_{XY}(X;Y)}{\vartheta_{Y}(X;Y)}$:
 \begin{itemize}
  \item [(1)] for $X\geq \frac{1}{5}$,
  $$-{\pi}\cdot \frac{1+{\nu}(X)}{1+{\mu}(X)}\leq \frac{\vartheta_{XY}(X;Y)}{\vartheta_{Y}(X;Y)} \leq -{\pi} \cdot\frac{1+\hat{\nu}(X)}{1+\hat{\mu}(X)};$$
  \item [(2)] for $0<X\leq \frac{1}{2}$,
  \begin{equation}\aligned\label{lem2.11}
  \frac{\frac{3}{4}{X}^2+2{\pi}^2 e^{-\frac{\pi}{X}}}{-\frac{1}{2}{X}^3+2\pi{X}^2e^{-\frac{\pi}{X}}}\leq \frac{\vartheta_{XY}(X;Y)}{\vartheta_{Y}(X;Y)} \leq \frac{\pi}{4{X}^2};
  \endaligned\end{equation}
\end{itemize}
  where $\mu(X),\:\hat{\mu}(X),\:\nu(X),\:\hat{\nu}(X)$ are defined in \eqref{duct}.
\end{lemma}
\begin{proof}
First, we prove the first item. By Lemma \ref{mono} and its proof, one has for $X\geq\frac{1}{5},Y\in\mathbb{R}$,
 \begin{equation}\aligned\nonumber
\lim_{Y\rightarrow 0}\frac{\vartheta_{XY}(X;Y)}{\vartheta_{Y}(X;Y)}\leq \frac{\vartheta_{XY}(X;Y)}{\vartheta_{Y}(X;Y)}\leq \lim_{Y\rightarrow \frac{1}{2}}\frac{\vartheta_{XY}(X;Y)}{\vartheta_{Y}(X;Y)}.
\endaligned\end{equation}
By \eqref{TXY} and L'Hospital's rule, one gets
\begin{equation}\aligned\label{nice}
 \lim_{Y\rightarrow 0}\frac{\vartheta_{XY}(X;Y)}{\vartheta_{Y}(X;Y)}
=-{\pi}\cdot \frac{1+{\nu}(X)}{1+{\mu}(X)},\:\:
 \lim_{Y\rightarrow \frac{1}{2}}\frac{\vartheta_{XY}(X;Y)}{\vartheta_{Y}(X;Y)}
=-{\pi}\cdot \frac{1+\hat{{\nu}}(X)}{1+\hat{{\mu}}(X)}.
\endaligned\end{equation}
Then, we prove the second item. For simplicity, we denote that
\begin{equation}\aligned\nonumber
 a_{1}(X)&:=\frac{3}{4}X^{2}\underset{n\in \mathbb{Z}}{\sum}e^{-\frac{\pi {n}^2}{X}}-3 \pi X \underset{n\in \mathbb{Z}}{\sum}{n}^2e^{-\frac{\pi {n}^2}{X}}+{\pi}^2\underset{n\in \mathbb{Z}}{\sum}{n}^4e^{-\frac{\pi {n}^2}{X}},\\
  a_{2}(X)&:=\pi X^{2}\underset{n\in \mathbb{Z}}{\sum}{n}^2e^{-\frac{\pi {n}^2}{X}}-\frac{1}{2} X^{3}\underset{n\in \mathbb{Z}}{\sum}e^{-\frac{\pi {n}^2}{X}}.\\
   \endaligned\end{equation}
 And denote that
 \begin{equation}\aligned\nonumber
   b_{1}(X)&:=\frac{3}{4}X^{2}\underset{n\in \mathbb{Z}}{\sum}e^{-\frac{\pi (n-\frac{1}{2})^2}{X}}-3 \pi X\underset{n\in \mathbb{Z}}{\sum}(n-\frac{1}{2})^2e^{-\frac{\pi (n-\frac{1}{2})^2}{X}}+{\pi}^2\underset{n\in \mathbb{Z}}{\sum}(n-\frac{1}{2})^4e^{-\frac{\pi (n-\frac{1}{2})^2}{X}},\\
   b_{2}(X)&:=\pi X^{2}\underset{n\in \mathbb{Z}}{\sum}(n-\frac{1}{2})^2e^{-\frac{\pi (n-\frac{1}{2})^2}{X}}-\frac{1}{2} X^{3}\underset{n\in \mathbb{Z}}{\sum}e^{-\frac{\pi (n-\frac{1}{2})^2}{X}},\\
    c_{1}(X)&:=\frac{3}{4}{X}^2+2{\pi}^2 e^{-\frac{\pi}{X}},\:\:\:\:\:\:\:\:\:\:\:\:\:\:\:\:\:\:\:\:\:\:\:\:\:\:c_{2}(X):=-\frac{1}{2}{X}^3+2\pi{X}^2e^{-\frac{\pi}{X}}.
 \endaligned\end{equation}
  By Lemma \ref{mono} and its proof, for $0<X\leq\frac{1}{2},Y\in\mathbb{R}$, one has
 \begin{equation}\aligned\label{calma}
\lim_{Y\rightarrow 0}\frac{\vartheta_{XY}(X;Y)}{\vartheta_{Y}(X;Y)}\leq \frac{\vartheta_{XY}(X;Y)}{\vartheta_{Y}(X;Y)}\leq \lim_{Y\rightarrow \frac{1}{2}}\frac{\vartheta_{XY}(X;Y)}{\vartheta_{Y}(X;Y)}.
\endaligned\end{equation}
  Then by Lemmas \ref{Lemmaair1}, one has
  \begin{equation}\aligned\label{cost3}
 \lim_{Y\rightarrow 0}\frac{\vartheta_{XY}(X;Y)}{\vartheta_{Y}(X;Y)}=\frac{a_{1}(X)}{a_{2}(X)},\:\:\lim_{Y\rightarrow \frac{1}{2}}\frac{\vartheta_{XY}(X;Y)}{\vartheta_{Y}(X;Y)}=\frac{b_{1}(X)}{b_{2}(X)}.
   \endaligned\end{equation}
  By \eqref{calma}-\eqref{cost3}, to prove \eqref{lem2.11}, it suffices to prove
  \begin{equation}\aligned\label{eye}
 \frac{c_{1}(X)}{c_2(X)}\leq \frac{a_{1}(X)}{a_{2}(X)},\:\: \frac{b_{1}(X)}{b_{2}(X)}\leq \frac{\pi}{4{X}^2}.
  \endaligned\end{equation}
    Notice that as $0<X\leq\frac{1}{2}$, $a_{2}(X),c_{2}(X)$ are negative, while $b_{2}(X)$ is positive. Hence \eqref{eye} is equivalent to
  \begin{equation}\aligned\nonumber
 a_{1}(X)c_{2}(X)-a_{2}(X)c_{1}(X)\geq0,\:\:\pi b_{2}(X)-4{X}^2 b_{1}(X)\geq 0.
  \endaligned\end{equation}
  By straightforward computations, as $0<X\leq\frac{1}{2}$, one gets
  \begin{equation}\aligned\nonumber
 a_{1}(X)c_{2}(X)-a_{2}(X)c_{1}(X)=3\pi{X}^4{e}^{-\frac{\pi}{X}}+(3\pi{X}^4-10{\pi}^2{X}^3){e}^{-\frac{2\pi}{X}}+T(X)\geq 0,
  \endaligned\end{equation}
   and
  \begin{equation}\aligned\nonumber
 \pi b_{2}(X)-4{X}^2 b_{1}(X)=(5\pi{X}^3-6{X}^4)e^{-\frac{\pi}{4X}}+\underset{n=2}{\overset{\infty}{\sum}}d(n;X)e^{-\frac{\pi(n-\frac{1}{2})^2}{X}}\geq 0.
   \endaligned\end{equation}
  Here,
  \begin{equation}\aligned\nonumber
 &T(X):=\underset{n=2}{\overset{\infty}{\sum}}(3\pi{X}^4+2{\pi}^2{X}^3+4{\pi}^3{X}^2{n}^4-4{\pi}^3{X}^2{n}^2-12{\pi}^2{X}^3{n}^2)e^{-\frac{\pi(n^2+1)}{X}}\\
 &\:\:\:\:\:\:\:\:\:\:\:\:\:\:\:\:\:\:\:\:\:+\underset{n=2}{\overset{\infty}{\sum}}(\frac{3}{2}\pi{X}^4{n}^2-{\pi}^2{X}^3{n}^4)e^{-\frac{\pi{n}^2}{X}},\\
 & d(n;X):=(2{\pi}^2{X}^2+24\pi{X}^3)\cdot(n-\frac{1}{2})^2-8{\pi}^2 {X}^2 (n-\frac{1}{2})^4 -\pi{X}^3-6{X}^4.
   \endaligned\end{equation}
 \end{proof}

\begin{lemma}\label{Lemmavar2}
Assume that $Y>0$, we have the upper and lower bounds for $\frac{\vartheta_{XXY}(X;Y)}{\vartheta_{Y}(X;Y)}:$\\
   \begin{itemize}
  \item [(1)] for $X\geq \frac{59}{250},$
  $${\pi}^2\cdot \frac{1+\hat{\omega}(X)}{1+\hat{\mu}(X)}\leq \frac{\vartheta_{XXY}(X;Y)}{\vartheta_{Y}(X;Y)} \leq {\pi}^2 \cdot\frac{1+\omega(X)}{1+\mu(X)};$$
  \item [(2)] for $0<X\leq \frac{1}{2}$,
  $$\frac{15}{4}X^{-2}(1-\frac{1}{3}\pi X^{-1}-\frac{1}{60}\pi^{2}X^{-2})\leq \frac{\vartheta_{XXY}(X;Y)}{\vartheta_{Y}(X;Y)} \leq \frac{15}{4}X^{-2}(1+\frac{1}{3}\pi X^{-1}+\frac{1}{60}\pi^{2}X^{-2});$$
\end{itemize}
where $\mu(X),\:\hat{\mu}(X),\:\omega(X),\:\hat{\omega}(X)$ are defined in \eqref{duct}.
\end{lemma}
\begin{proof}
Firstly, we prove the first item. By \eqref{stormy}, notice that $\frac{\vartheta_{XXY}(X;Y)}{\vartheta_{Y}(X;Y)}$ has a period of 1 with respect to $Y$ and  an axis of symmetry at $Y=\frac{1}{2}$. It reduces to consider that for $Y \in[0,\frac{1}{2}]$. Then by Lemma\:\ref{monoo}, for $X\geq\frac{59}{250}$ and $Y\in[0,\frac{1}{2}]$, one has
\begin{equation}\aligned\nonumber
\lim_{Y\rightarrow \frac{1}{2}}\frac{\vartheta_{XXY}(X;Y)}{\vartheta_{Y}(X;Y)}\leq \frac{\vartheta_{XXY}(X;Y)}{\vartheta_{Y}(X;Y)}\leq \lim_{Y\rightarrow 0}\frac{\vartheta_{XXY}(X;Y)}{\vartheta_{Y}(X;Y)}.
\endaligned\end{equation}
Then, by \eqref{stormy} and L'Hospital's rule, one has
\begin{equation}\aligned\label{highh1}
 \lim_{Y\rightarrow \frac{1}{2}}\frac{\vartheta_{XXY}(X;Y)}{\vartheta_{Y}(X;Y)}
={\pi}^2\cdot \frac{1+\hat{\omega}(X)}{1+\hat{\mu}(X)},\:\:
  \lim_{Y\rightarrow 0}\frac{\vartheta_{XXY}(X;Y)}{\vartheta_{Y}(X;Y)}
={\pi}^2\cdot \frac{1+{\omega}(X)}{1+{\mu}(X)}.
\endaligned\end{equation}
Next, we prove the second item. By \eqref{Poisson}, one has
\begin{equation}\aligned\label{adjust2}
\vartheta_{Y}(X;Y)=&2\pi X^{-\frac{3}{2}}\underset{n\in \mathbb{Z}}{\sum}(n-Y)e^{-\frac{-\pi(n-Y)^2}{X}},\\
\vartheta_{XXY}(X;Y)=&\frac{15}{2}\pi X^{-\frac{7}{2}}\underset{n\in \mathbb{Z}}{\sum}(n-Y)e^{-\frac{-\pi(n-Y)^2}{X}}-10{\pi}^2 X^{-\frac{9}{2}}\underset{n\in \mathbb{Z}}{\sum}(n-Y)^3e^{-\frac{-\pi(n-Y)^2}{X}}\\
&\:\:+2{\pi}^3 X^{-\frac{11}{2}}\underset{n\in \mathbb{Z}}{\sum}(n-Y)^5e^{-\frac{-\pi(n-Y)^2}{X}}.\\
\endaligned\end{equation}
Through \eqref{adjust2}, the quotient of $\vartheta_{XXY}(X;Y)$ and $\vartheta_{Y}(X;Y)$ has the following expression
\begin{equation}\aligned\label{LLL}
\frac{\vartheta_{XXY}(X;Y)}{\vartheta_{Y}(X;Y)}=\frac{15}{4}X^{-2} -5\pi {X}^{-3}\cdot \frac{\underset{n\in \mathbb{Z}}{\sum}(n-Y)^3 e^{-\frac{-\pi(n-Y)^2}{X}}}
{\underset{n\in \mathbb{Z}}{\sum}(n-Y)e^{-\frac{-\pi(n-Y)^2}{X}}} +{\pi}^2X^{-4}\cdot \frac{\underset{n\in \mathbb{Z}}{\sum}(n-Y)^5 e^{-\frac{-\pi(n-Y)^2}{X}}}
{\underset{n\in \mathbb{Z}}{\sum}(n-Y)e^{-\frac{-\pi(n-Y)^2}{X}}}.
\endaligned\end{equation}
Therefore, the second item follows by Lemmas \ref{Lemmafour}-\ref{Lemmasixteen} and \eqref{LLL}.
\end{proof}

\begin{lemma}\label{Lemmavar1}
Assume that $Y>0,k\in \mathbb{N^{+}}$, it holds that
\begin{itemize}
  \item [(1)] for $X\geq \frac{1}{5}$,
  $$\left|\frac{\vartheta_{XXY}(X;kY)}{\vartheta_{Y}(X;Y)}\right| \leq k\pi^2 \cdot \frac{1+\omega(X)}{1-\mu(X)};$$
  \item [(2)] for $0<X\leq \frac{1}{2}$,
  $$
  \left|\frac{\vartheta_{XXY}(X;kY)}{\vartheta_{Y}(X;Y)} \right|\leq \frac{15k}{4\pi}X^{-2}(1+\frac{1}{3}\pi X^{-1}+\frac{1}{60}\pi^{2}X^{-2})e^{\frac{\pi}{4X}}; $$
\end{itemize}
where $\mu(X),\:\omega(X)$ are defined in \eqref{duct}.
\end{lemma}
\begin{proof}
Firstly, we prove the first item. Based on \eqref{TXY}, for $X\geq \frac{1}{5}$, one has
 \begin{equation}\aligned\label{XXY}
\left|\frac{\vartheta_{XXY}(X;kY)}{\vartheta_{Y}(X;Y)}\right|
&={\pi}^2 \left|\frac{\sin(2k\pi Y)}{\sin(2\pi Y)} \frac{\overset{\infty}{\underset{n=1}{\sum}}n^5 e^{-\pi (n^2-1) X} \frac{\sin(2nk\pi Y)}{\sin(2k\pi Y)}}{\overset{\infty}{\underset{n=1}{\sum}}n e^{-\pi (n^2-1) X} \frac{\sin(2n\pi Y)}{\sin(2\pi Y)}}\right| \\
&\leq k {\pi}^2  \cdot \frac{1+\overset{\infty}{\underset{n=2}{\sum}}n^6 e^{-\pi (n^2-1) X} }{1-\overset{\infty}{\underset{n=2}{\sum}}n^2 e^{-\pi (n^2-1) X} }=k\pi^2 \cdot \frac{1+\omega(X)}{1-\mu(X)}.
\endaligned\end{equation}
Notice that $1-\mu(X)>0,\:for\: X\geq \frac{1}{5}$. In \eqref{XXY}, we used the following inequality:
\begin{equation}\aligned\nonumber
\left|\frac{\sin(kx)}{\sin(x)} \right|  \leq k,\:\:for\:x\:\in \mathbb{R},\:k\in \mathbb{N}^{+}.
\endaligned\end{equation}
Next, we prove the second item. By a simple deformation, one has
\begin{equation}\aligned\label{GGG}
\left|\frac{\vartheta_{XXY}(X;kY)}{\vartheta_{Y}(X;Y)} \right|=\left|\frac{\vartheta_{XXY}(X;kY)}{\vartheta_{Y}(X;kY)} \right|\cdot \left|\frac{\vartheta_{Y}(X;kY)}{\vartheta_{Y}(X;Y)} \right|.
\endaligned\end{equation}
Here, by Lemma \ref{Lemmavar2} and \eqref{GGG}, one gets
\begin{equation}\aligned\label{has}
\left|\frac{\vartheta_{XXY}(X;kY)}{\vartheta_{Y}(X;kY)}\right|\leq \frac{15}{4}X^{-2}(1+\frac{1}{3}\pi X^{-1}+\frac{1}{60}\pi^{2}X^{-2}).
\endaligned\end{equation}
Thus, the second item is derived from \eqref{GGG}, \eqref{has} and Lemma \ref{Lemmamoist}.
\end{proof}

\subsection{Two auxiliary Lemmas to estimate $\frac{\vartheta_{XY}(X;Y)}{\vartheta_{Y}(X;Y)}$}
The first auxiliary Lemma will show the monotonicity of $\frac{\vartheta_{XY}(X;Y)}{\vartheta_{Y}(X;Y)}$ with respect to $Y$. The second auxiliary Lemma to estimate $\frac{\vartheta_{XY}(X;Y)}{\vartheta_{Y}(X;Y)}$ will give its different expressions as $Y$ approaches to $0$ and $\frac{1}{2}$ separately.

\begin{lemma}\label{mono}
Assume $X>0,\:k\in\mathbb{Z}$,\:it holds that
\begin{itemize}
    \item [(1)] $\frac{\partial}{\partial{Y}}    \big(\frac{\vartheta_{XY}(X;Y)}{\vartheta_{Y}(X;Y)}\big)\geq0$, for $Y\in[k,k+\frac{1}{2}];$
    \item [(2)] $\frac{\partial}{\partial{Y}}    \big(\frac{\vartheta_{XY}(X;Y)}{\vartheta_{Y}(X;Y)}\big)\leq0$, for $Y\in[k+\frac{1}{2},k+1]$.
 \end{itemize}
\end{lemma}
\begin{proof}
For simplicity, we denote that
\begin{equation}\aligned\label{la}
\mathcal{N}(X;Y):=-\frac{3}{2X}+\frac{\pi}{{X}^2} \cdot \frac{\underset{{n\in\mathbb{Z}}}{\sum}(n-Y)^3 e^{-\frac{\pi(n-Y)^2}{X}}}{\underset{{n\in\mathbb{Z}}}{\sum}(n-Y)e^{-\frac{\pi(n-Y)^2}{X}}}.
\endaligned\end{equation}
Based on\:\eqref{Poisson}, one gets
\begin{equation}\aligned\label{monotone1}\nonumber
\frac{\vartheta_{XY}(X;Y)}{\vartheta_{Y}(X;Y)}=\mathcal{N}(X;Y).
\endaligned\end{equation}
A direct checking shows that
\begin{equation}\aligned\label{monotone2}\nonumber
\mathcal{N}(X;Y)=\mathcal{N}(X;Y+1),\:\mathcal{N}(X;Y)=\mathcal{N}(X;1-Y).
\endaligned\end{equation}
The monotonicity of $\mathcal{N}(X;Y)$ with respect to $Y$ as $X>0,\:Y\in[\frac{1}{2},1]$ is the opposite of that as $X>0, Y\in[0,\frac{1}{2}]$ if monotonicity exists. Therefore it is reduced to the monotonicity of $\mathcal{N}(X;Y)$ with respect to $Y$ for $X>0,\:Y\in[0,\frac{1}{2}]$. By the proof of Lemma 2.4 in \cite{Luo2023}, one has
\begin{equation}\aligned\label{cite}
\frac{\partial}{\partial{Y}}\mathcal{N}(X;Y) \geq 0,\:for \:X\geq \frac{21}{100},\:Y\in[0,\frac{1}{2}].
\endaligned\end{equation}
By \eqref{la}, one has
\begin{equation}\aligned\label{monotonee}
\frac{\partial}{\partial{Y}}\mathcal{N}(X;Y)=\frac{\pi}{{X}^2}\cdot \frac{\partial}{\partial{Y}}\bigg(\frac{\underset{{n\in\mathbb{Z}}}{\sum}(n-Y)^3 e^{-\frac{\pi(n-Y)^2}{X}}}{\underset{{n\in\mathbb{Z}}}{\sum}(n-Y)e^{-\frac{\pi(n-Y)^2}{X}}}    \bigg).
\endaligned\end{equation}
Then due to \eqref{monotonee} and the proof of Lemma 2.6 in \cite{Luo2023}, one has
\begin{equation}\aligned\label{citee}
\frac{\partial}{\partial{Y}}\mathcal{N}(X;Y) \geq 0,\:for\:\: 0<X<\frac{1}{2},\:Y\in[0,\frac{1}{2}].
\endaligned\end{equation}
Thus, \eqref{cite} and \eqref{citee} yield the result.
\end{proof}

 Next, we demonstrate the second auxiliary Lemma to estimate $\frac{\vartheta_{XY}(X;Y)}{\vartheta_{Y}(X;Y)}$.

\begin{lemma}[The expressions for $\frac{\vartheta_{XY}(X;Y)}{\vartheta_{Y}(X;Y)}$ as $Y$ approaches to $0$ and $\frac{1}{2}$ separately]\label{Lemmaair1}
 \begin{equation}\aligned\nonumber
 &\lim_{Y\rightarrow 0}\frac{\vartheta_{XY}(X;Y)}{\vartheta_{Y}(X;Y)}=\frac{\frac{3}{4}X^{2}\underset{n\in \mathbb{Z}}{\sum}e^{-\frac{\pi {n}^2}{X}}-3 \pi X \underset{n\in \mathbb{Z}}{\sum}{n}^2e^{-\frac{\pi {n}^2}{X}}+{\pi}^2\underset{n\in \mathbb{Z}}{\sum}{n}^4e^{-\frac{\pi {n}^2}{X}}}{\pi X^{2}\underset{n\in \mathbb{Z}}{\sum}{n}^2e^{-\frac{\pi {n}^2}{X}}-\frac{1}{2} X^{3}\underset{n\in \mathbb{Z}}{\sum}e^{-\frac{\pi {n}^2}{X}}}.\\
  &\lim_{Y\rightarrow \frac{1}{2}}\frac{\vartheta_{XY}(X;Y)}{\vartheta_{Y}(X;Y)}=\frac{\frac{3}{4}X^{2}\underset{n\in \mathbb{Z}}{\sum}e^{-\frac{\pi (n-\frac{1}{2})^2}{X}}-3 \pi X\underset{n\in \mathbb{Z}}{\sum}(n-\frac{1}{2})^2e^{-\frac{\pi (n-\frac{1}{2})^2}{X}}+{\pi}^2\underset{n\in \mathbb{Z}}{\sum}(n-\frac{1}{2})^4e^{-\frac{\pi (n-\frac{1}{2})^2}{X}}}{\pi X^{2}\underset{n\in \mathbb{Z}}{\sum}(n-\frac{1}{2})^2e^{-\frac{\pi (n-\frac{1}{2})^2}{X}}-\frac{1}{2} X^{3}\underset{n\in \mathbb{Z}}{\sum}e^{-\frac{\pi (n-\frac{1}{2})^2}{X}}}.
\endaligned\end{equation}
  \end{lemma}
 \begin{proof}
First, we introduce Jacobi theta functions of the second, the third and the fourth type defined as follows:
\begin{equation}\aligned\label{JacobiJJJ}
\vartheta_{2}(X)=\underset{n\in\mathbb{Z}}{\sum}e^{-\pi(n-\frac{1}{2})^2X},\:\:\vartheta_{3}(X)=\underset{n\in\mathbb{Z}}{\sum}e^{-\pi{n}^2X},\:\:\vartheta_{4}(X)=\underset{n\in\mathbb{Z}}{\sum}(-1)^{n}e^{-\pi n^2 X}
\endaligned\end{equation}
Notice that $\vartheta_{3}(X)$ satisfies the following transformation property (Conway-Sloane \cite{Con1999}):
\begin{equation}\aligned\label{Fourierr}
\vartheta_{3}(X)=\frac{1}{\sqrt{X}}\vartheta_{3}\Big(\frac{1}{X}\Big),\:\:\vartheta_{4}(X)=\frac{1}{\sqrt{X}}\vartheta_{2}(\frac{1}{X})\:\:for\:\: X>0.
\endaligned\end{equation}
Then, by L'Hospital's rule, \eqref{TXY} and \eqref{JacobiJJJ}-\eqref{Fourierr}, one has
\begin{equation}\aligned\label{falll}
\lim_{Y\rightarrow 0}\frac{\vartheta_{XY}(X;Y)}{\vartheta_{Y}(X;Y)}
=-\pi \cdot\frac{ \underset{n\in\mathbb{Z}}{\sum} {n}^4 e^{-\pi n^2 X}}{\underset{n\in\mathbb{Z}}{\sum} n^2\: e^{-\pi n^2 X}}
=\frac{\vartheta_{3}''(X)}{\vartheta_{3}'(X)}=\frac{\frac{3}{4}X^{2}\vartheta_{3}(\frac{1}{X})+3X\vartheta_{3}^{'}(\frac{1}{X})+\vartheta_{3}^{''}(\frac{1}{X})}
{-\frac{1}{2}X^{3}\vartheta_{3}(\frac{1}{X})-X^{2}\vartheta_{3}^{'}(\frac{1}{X})},
\endaligned\end{equation}
\begin{equation}\aligned\label{fall}
\lim_{Y\rightarrow \frac{1}{2}}\frac{\vartheta_{XY}(X;Y)}{\vartheta_{Y}(X;Y)}
=-\pi \cdot\frac{ \underset{n\in\mathbb{Z}}{\sum} {n}^4 e^{-\pi n^2 X}(-1)^{n}}{\underset{n\in\mathbb{Z}}{\sum} n^2\: e^{-\pi n^2 X}(-1)^{n}}
=\frac{\vartheta_{4}''(X)}{\vartheta_{4}'(X)}=\frac{\frac{3}{4}X^{2}\vartheta_{2}(\frac{1}{X})+3X\vartheta_{2}^{'}(\frac{1}{X})+\vartheta_{2}^{''}(\frac{1}{X})}
{-\frac{1}{2}X^{3}\vartheta_{2}(\frac{1}{X})-X^{2}\vartheta_{2}^{'}(\frac{1}{X})}.\\
\endaligned\end{equation}
Therefore, the result is derived from \eqref{JacobiJJJ} and \eqref{falll}-\eqref{fall}.
\end{proof}

\subsection{Four auxiliary Lemmas to estimate $\frac{\vartheta_{XXY}(X;Y)}{\vartheta_{Y}(X;Y)}$}

The first auxiliary Lemma will demonstrate a monotonicity result of $ \frac{\vartheta_{XXY}(X;Y)}{\vartheta_{Y}(X;Y)}$ with respect to $Y$. The second and third auxiliary Lemmas will give the estimates of two distinct quotients in \eqref{LLL}, where the second Lemma is provided by Luo-Wei \cite{Luo2023} and the third Lemma is inspired by the second Lemma. And the fourth Lemma is useful to prove the third auxiliary Lemma.

\begin{lemma}\label{monoo}
Assume $X\geq\frac{59}{250},\:Y \in[k,k+\frac{1}{2}],\:k\in\mathbb{Z}$,\:it holds that
\begin{equation}\nonumber\aligned\label{high}
\frac{\partial}{\partial{Y}} \Big( \frac{\vartheta_{XXY}(X;Y)}{\vartheta_{Y}(X;Y)}\Big)\leq 0.
\endaligned\end{equation}
\end{lemma}

\begin{proof}
Based on \eqref{TXY}, one has
\begin{equation}\label{stormy}
\frac{\vartheta_{XXY}(X;Y)}{\vartheta_{Y}(X;Y)}={\pi}^2\cdot\frac{\underset{n=1}{\overset{\infty}{\sum}}n^{5}e^{-\pi{n}^2{X}}\sin(2n\pi{Y})}{\underset{n=1}{\overset{\infty}{\sum}}ne^{-\pi{n}^2{X}}\sin(2n\pi{Y})}.
\end{equation}
Notice that $\frac{\vartheta_{XXY}(X;Y)}{\vartheta_{Y}(X;Y)}$ has a period of 1 with respect to $Y$ and  an axis of symmetry at $Y=\frac{1}{2}$. It reduces to consider that for $Y \in[0,\frac{1}{2}]$.
By \eqref{stormy}, one has
\begin{equation}\label{zz}
\frac{\partial}{\partial{Y}} \Big(\frac{\vartheta_{XXY}(X;Y)}{\vartheta_{Y}(X;Y)}\Big)=\frac{G(X;Y)}{B^2(X;Y)},
\end{equation}
where
\begin{equation}\aligned
G(X;Y):=\underset{n=1}{\overset{\infty}{\sum}}\:\underset{m=1}{\overset{\infty}{\sum}}G_{n,m}(X;Y),\\
\endaligned\end{equation}
\begin{equation}\aligned\label{ht1}
G_{n,m}(X;Y):={\pi}^{2}nm(n^4-m^4)e^{-\pi(n^2+m^2-2)X}\Big(\frac{\sin(2n\pi{Y})}{\sin(2\pi{Y})}\Big)'
\Big(\frac{\sin(2m\pi{Y})}{\sin(2\pi{Y})}\Big),
\endaligned\end{equation}
\begin{equation}
B(X;Y):=\underset{m=1}{\overset{\infty}{\sum}}m{e}^{-\pi(m^2-1)X}\frac{\sin(2m\pi{Y})}{\sin(2\pi{Y})}.
\end{equation}
Further, we split the double sum $G(X;Y)$ into four parts as follows
\begin{equation}\nonumber
\underset{n=1,2}{\sum}\:\underset{m=1,2}{\sum}+\underset{n=1,2,\:m\geq 3}{\sum}+\underset{m=1,2,\:n\geq 3}{\sum}+\underset{n\geq 3,m\geq3}{\sum}.
\end{equation}
By a direct computation,\:one has
\begin{equation}\label{ht2}
\underset{n=1}{\overset{2}{\sum}}\:\underset{m=1}{\overset{2}{\sum}}G_{n,m}(X;Y)=-120{\pi}^{3}e^{-3\pi{X}}\sin(2\pi{Y}).
\end{equation}
By Lemma 2.7 of \cite{Luo2023}, it holds that
\begin{equation}\aligned\label{ht3}
\left|  \frac{1}{\sin(2\pi{Y})} (\frac{\sin(2n\pi{Y})}{\sin(2\pi{Y})})' \right|\leq \frac{2\pi}{3}(n-1)n(n+1),\:\:for \:\:Y\in \mathbb{R}.
\endaligned\end{equation}
Then by \eqref{ht3}, one has
\begin{equation}\aligned\label{ht4}
\left|\frac{G_{n,m}(X;Y)}{-120{\pi}^{3}e^{-3\pi{X}}\sin(2\pi{Y})}\right|&=\left|  \frac{1}{120\pi}  nm(n^4-m^4)e^{-\pi(n^2+m^2-5)X} \frac{1}{sin(2\pi{Y})} (\frac{\sin(2n\pi{Y})}{\sin(2\pi{Y})})'\frac{\sin(2m\pi{Y})}{\sin(2\pi{Y})} \right|\\
&\leq \frac{1}{180}n^2{m}^2\left| n^4-m^4 \right|\cdot\left| n^2-1 \right|e^{-\pi(n^2+m^2-5)X}.
\endaligned\end{equation}
By \eqref{ht1}-\eqref{ht4},
\begin{equation}\aligned\label{ht5}
\frac{G(X;Y)}{-120{\pi}^{3}e^{-3\pi{X}}\sin(2\pi{Y})}&\geq1-\frac{1}{180} \underset{n\geq3\:\mathrm{or}\:m\geq 3}{\sum}n^2{m}^2\left| n^4-m^4 \right|\cdot\left| n^2-1 \right|e^{-\pi(n^2+m^2-5)X}.\\
\endaligned\end{equation}
Here,
\begin{equation}\aligned\label{ht6}\nonumber
\underset{n\geq3\:\mathrm{or}\:m\geq 3}{\sum}=\underset{n=1,2,\:m\geq 3}{\sum}+\underset{m=1,2,\:n\geq 3}{\sum}+\underset{n\geq 3,m\geq3}{\sum}.
\endaligned\end{equation}
The right-hand side of \eqref{ht5} is positive if $X>\frac{59}{250}$.
Thus, \eqref{zz} and \eqref{ht5} yield the result.
\end{proof}

\begin{lemma}[Luo-Wei \cite{Luo2023}]\label{Lemmafour}
\begin{equation}\nonumber\aligned\label{Gfff}
\underset{X\in(0,\frac{1}{2}],Y\in \mathbb{R}}{\sup}\left|   \frac{\sum_{n\in\mathbb{Z}}(n-Y)^3 e^{-\frac{\pi(n-Y)^2}{X}}}{\sum_{n\in \mathbb{Z}}(n-Y)e^{-\frac{\pi(n-Y)^2}{X}}}        \right|\leq\frac{1}{4}.
\endaligned\end{equation}
\end{lemma}

Having provided the estimate in Lemma \ref{Lemmafour} given by $\mathrm{Luo}$-$\mathrm{Wei}$ \cite{Luo2023}, we proceed to give an estimate of a similar expression in Lemma \ref{Lemmasixteen}, which is also availed in the proof of Lemma \ref{Lemmavar2}.

\begin{lemma}\label{Lemmasixteen}
\begin{equation}\nonumber\aligned\label{Gfff}
\underset{X\in(0,\frac{1}{2}],Y\in \mathbb{R}}{\sup}\left|   \frac{\sum_{n\in\mathbb{Z}}(n-Y)^5 e^{-\frac{\pi(n-Y)^2}{X}}}{\sum_{n\in \mathbb{Z}}(n-Y)e^{-\frac{\pi(n-Y)^2}{X}}}        \right|\leq\frac{1}{16}.
\endaligned\end{equation}
\end{lemma}
\begin{proof}
Let $a=\frac{1}{X}$\:and\:$f(a;Y):=\frac{\underset{n\in \mathbb{Z}}{\sum}(n-Y)^5 e^{-a\pi(n-Y)^2}}{\underset{n\in \mathbb{Z}}{\sum}(n-Y) e^{-a\pi(n-Y)^2}}$, then $a\geq 2$. It has the following properties:
\begin{equation}\nonumber\aligned\label{property}
f(a;Y+1)=f(a;Y),\:f(a;1-Y)=f(a;Y).
\endaligned\end{equation}
Then it reduces to consider $f(a,Y)$ for $a\geq 2$ and $Y\in[0,\frac{1}{2}]$. It suffices to prove that
\begin{equation}\nonumber\aligned
\underset{a\geq 2,Y\in [0,\frac{1}{2}]}{\sup}\left|f(a;Y)\right|\leq \frac{1}{16}.
\endaligned\end{equation}
By the proof of Lemma 2.6 (Luo-Wei \cite{Luo2023}), one has
\begin{equation}\aligned\label{cup}
f_{Y}(a;Y)>0 \:\:for\:\: a\geq 2, Y\in [0,\frac{1}{2}].
\endaligned\end{equation}
By \eqref{cup}, one has
\begin{equation}\aligned\label{yep}
\underset{Y\in[0,\frac{1}{2}]}{\max}\left| f(a;Y) \right|=\max\left\{\left| f(a;0) \right|,\left| f(a;\frac{1}{2}) \right|\right\}.
\endaligned\end{equation}
By L'\:Hospital's rule,
\begin{equation}\aligned
\left| f(a;0) \right|&=\left| \frac{\underset{n\in \mathbb{Z}}{\sum}(2 a\pi n^6-5n^4) e^{-a \pi n^2}}
{\underset{n\in \mathbb{Z}}{\sum}(2 a \pi n^2-1)e^{-a \pi n ^2}} \right|\leq \frac{4a \pi\underset{n=1}{\overset{\infty}{\sum}} n^6 e^{-a \pi n^2}}
{1-4a \pi\underset{n=1}{\overset{\infty}{\sum}}n^2 e^{-a \pi n^2}} \leq 0.05 \:\:for\:\: a\:\geq 2.
\endaligned\end{equation}
And
\begin{equation}\aligned\label{lighthouse}
\left| f(a;\frac{1}{2}) \right|&=\left| \frac{\underset{n\in \mathbb{Z}}{\sum}[2 a\pi (n-\frac{1}{2})^6-5(n-\frac{1}{2})^4] e^{-a \pi(n-\frac{1}{2})^2}}
{\underset{n\in \mathbb{Z}}{\sum}[2 a \pi (n-\frac{1}{2})^2-1]e^{-a \pi(n-\frac{1}{2})^2}} \right|=\frac{1}{16}\left|\frac{a\pi-10}{a\pi-2} \right|\cdot \left|  \frac{1+\sigma_{a1}}{1+\sigma_{a2}} \right|.
\endaligned\end{equation}
Here
\begin{equation}\aligned\label{vv}
\sigma_{a1}:= \underset{n=2}{\overset{\infty}{\sum}}\frac{2 a\pi (n-\frac{1}{2})^6-5(n-\frac{1}{2})^4 }
{\frac{a\pi}{32}-\frac{5}{16}} e^{-a \pi(n^2-n)},\:\:
\sigma_{a2}:=\underset{n=2}{\overset{\infty}{\sum}}\frac{2 a\pi (n-\frac{1}{2})^2-1 }
{\frac{a\pi}{2}-1} e^{-a \pi(n^2-n)}.\\
\endaligned\end{equation}
Thus, \eqref{yep}-\eqref{lighthouse} and Lemma \ref{Lemmaauxi} yield the result.
\end{proof}

\begin{lemma}\label{Lemmaauxi}
Assume that $ a\geq 2$, it holds that
\begin{equation}\nonumber\aligned
\left|\frac{a\pi-10}{a\pi-2} \right|\cdot \left|  \frac{1+\sigma_{a1}}{1+\sigma_{a2}} \right|\leq 1,
\endaligned\end{equation}
where $\sigma_{a1}$ and $\sigma_{a2}$ are defined in \eqref{vv}.
\end{lemma}
\begin{proof}
We divide the interval $a\in[2,\infty)$ into the following  three cases, where $0<\varepsilon<\frac{1}{10}$.\\
Case one: $a\in[2,\frac{10}{\pi})$. In this case, we denote that
\begin{equation}\aligned\nonumber
h_{1}(a):=\frac{10-a\pi-32\underset{n=2}{\overset{\infty}{\sum}}[2 a\pi (n-\frac{1}{2})^6-5(n-\frac{1}{2})^4] e^{-a \pi (n^2-n)}}
{a\pi-2+2\underset{n=2}{\overset{\infty}{\sum}}[2 a\pi (n-\frac{1}{2})^2-1] e^{-a \pi (n^2-n)}}.
\endaligned\end{equation}
Notice that $h_{1}(a)$ is decreasing as $a\in[2,\frac{10}{\pi})$. Then
\begin{equation}\aligned\nonumber
\left|\frac{a\pi-10}{a\pi-2} \right|\cdot \left|  \frac{1+\sigma_{a1}}{1+\sigma_{a2}} \right|
=h_{1}(a)\leq h_{1}(2)<1.
\endaligned\end{equation}
Case two: $a\in(\frac{10}{\pi}-\varepsilon,\frac{10}{\pi}+\varepsilon)$. In this case, one has
\begin{equation}\aligned\label{Mars1}\nonumber
\underset{a\rightarrow \frac{10}{\pi}}{{\lim}}\left|\frac{a\pi-10}{a\pi-2} \right|\cdot \left|  \frac{1+\sigma_{a1}}{1+\sigma_{a2}} \right|
\endaligned\end{equation}
\begin{equation}\aligned\nonumber
=&\underset{a\rightarrow \frac{10}{\pi}}{{\lim}}\left|\frac{a\pi-10+32(a\pi-10)\underset{n=2}{\overset{\infty}{\sum}}[2 a\pi (n-\frac{1}{2})^6-5(n-\frac{1}{2})^4] e^{-a \pi (n^2-n)}}
{a\pi-2+2\underset{n=2}{\overset{\infty}{\sum}}[2 a\pi (n-\frac{1}{2})^2-1] e^{-a \pi (n^2-n)}}\right|=0.
\endaligned\end{equation}
Case three: $a>\frac{10}{\pi}$. In this case, we denote that
\begin{equation}\aligned\nonumber
h_{2}(a):=\frac{a\pi-10+32\underset{n=2}{\overset{\infty}{\sum}}[2 a\pi (n-\frac{1}{2})^6-5(n-\frac{1}{2})^4] e^{-a \pi (n^2-n)}}
{a\pi-2+2\underset{n=2}{\overset{\infty}{\sum}}[2 a\pi (n-\frac{1}{2})^2-1] e^{-a \pi (n^2-n)}}.
\endaligned\end{equation}
Notice that $h_{2}(a)$ is increasing as $a>\frac{10}{\pi}$ and $\underset{a\rightarrow +\infty}{{\lim}}h_{2}(a)=1$. Then
\begin{equation}\aligned\nonumber
\left|\frac{a\pi-10}{a\pi-2} \right|\cdot \left|  \frac{1+\sigma_{a1}}{1+\sigma_{a2}} \right|
=h_{2}(a)\leq 1.
\endaligned\end{equation}
 \end{proof}




\section{The transversal monotonicity}
\setcounter{equation}{0}
Define the right boundary of the  fundamental domain\:$\mathcal{D}_{\mathcal{G}}$ given by\:\eqref{Fd3} as follows
\begin{equation}\aligned\label{boundary}
\Gamma:=\{
z\in\mathbb{H}: \Re(z)=\frac{1}{2},\; \Im(z)\geq\frac{\sqrt3}{2}
\}.
\endaligned\end{equation}

In this Section, we aim to establish that
\begin{theorem}\label{2Th1} Assume that $\alpha\geq\frac{3}{2}$. Then
\begin{equation}\aligned\nonumber
\min_{z\in\mathbb{H}} \mathcal{R}(\alpha;z)
=\min_{z\in\overline{\mathcal{D}_{\mathcal{G}}}}\mathcal{R}(\alpha;z)
=\min_{z\in\Gamma}\mathcal{R}(\alpha;z).
\endaligned\end{equation}
\end{theorem}

By the group invariance (Lemma \ref{INR}), one has
\begin{equation}\aligned\label{Domain1}
\min_{z\in\mathbb{H}} \mathcal{R}(\alpha;z)
=\min_{z\in\overline{\mathcal{D}_{\mathcal{G}}}}\mathcal{R}(\alpha;z).
\endaligned\end{equation}

With the help of \eqref{Domain1}, it remains to prove the minimization on the fundamental domain can be reduced to its right boundary $\Gamma$, which is based on the following monotonicity result. The rest of this Section is devoted to the proof of Theorem \ref{2Th2}.

\begin{theorem}\label{2Th2} Assume that $\alpha\geq\frac{3}{2}$.  Then
\begin{equation}\aligned\nonumber
\frac{\partial}{\partial{x}}\mathcal{R}(\alpha;z)<0,\;\;for\ z\in\mathcal{D}_{\mathcal{G}}.
\endaligned\end{equation}
\end{theorem}

The function $\mathcal{R}$ is related to the Theta function closely (see Lemma \ref{respect}). Hence we shall start with the analysis of the Theta function. The following Lemma \ref{Lemmatree} provided by $\mathrm{Luo}$-$\mathrm{Wei}$ \cite{Luo2022} introduces an exponential expansion of the Theta function, which is useful in our proofs.

\begin{lemma}\label{Lemmatree}$(\mathrm{Luo}$-$\mathrm{Wei}$\:\cite{Luo2022}$)$ We have the following  exponential expansion of Theta function:
\begin{equation}\aligned\label{PPP3}
\theta (\alpha;z)=2\sqrt{\frac{y}{\alpha}}\sum_{n=1}^\infty e^{-\alpha \pi y n^2}\vartheta(\frac{y}{\alpha};nx)+\sqrt{\frac{y}{\alpha}}\vartheta(\frac{y}{\alpha};0).
\endaligned\end{equation}
\end{lemma}

Based on Lemmas \ref{respect} and \ref{Lemmatree}, we shall give an exponential expansion of derivative of the function $\mathcal{R}$ with respect to $x$ in Lemma \ref{part1}, whose proof is straightforward and will be omitted. Before that, to distinguish clearly the major terms and  the error terms in the expression of $\frac{\partial}{\partial x}\mathcal{R}(\alpha;z)$, we denote that

\begin{equation}\aligned\label{phi1}
\Phi_{\alpha,A}(z):=&{\pi}{\alpha}^4 y^2 +{\alpha}^3 y +\frac{3}{4\pi}{\alpha}^2
+(\frac{3}{\pi}{\alpha}y  +2{\alpha}^2 y^2 )\frac{\vartheta_{XY}(\frac{y}{\alpha};x)}{\vartheta_{Y}(\frac{y}{\alpha};x)}
+ \frac{1}{\pi}y^2 \frac{\vartheta_{XXY}(\frac{y}{\alpha};x)}{\vartheta_{Y}(\frac{y}{\alpha};x)}.\\
\endaligned\end{equation}
\begin{equation}\aligned\label{phi2}
\Phi_{\alpha,B}^{1}(z)&:=\sum_{n=2}^\infty ({\pi}{\alpha}^4 y^2 n^5+{\alpha}^3 y n^3+\frac{3}{4\pi}{\alpha}^2 n)e^{-\alpha\pi y(n^2-1)}\frac{\vartheta_{Y}(\frac{y}{\alpha};nx)}{\vartheta_{Y}(\frac{y}{\alpha};x)},\:\:\:\:\:\:\:\:\:\:\:\:\:\:\:\:\:\:\:\:\:\:\:\:\:\:\\
\Phi_{\alpha,B}^{2}(z)&:=\sum_{n=2}^\infty (\frac{3}{\pi}{\alpha}y  n+2{\alpha}^2 y^2 n^3)e^{-\alpha\pi y(n^2-1)}\frac{\vartheta_{XY}(\frac{y}{\alpha};nx)}{\vartheta_{Y}(\frac{y}{\alpha};x)},\\
\Phi_{\alpha,B}^{3}(z)&:=\sum_{n=2}^\infty \frac{1}{\pi}y^2 ne^{-\alpha\pi y(n^2-1)}\frac{\vartheta_{XXY}(\frac{y}{\alpha};nx)}{\vartheta_{Y}(\frac{y}{\alpha};x)}.\\
\endaligned\end{equation}

\begin{lemma}\label{part1} We have the following  identity for derivative of the function $\mathcal{R}$ with respect to $x$:
\begin{equation}\nonumber\aligned
-\frac{\partial}{\partial x} \mathcal{R}(\alpha;z)
=\mathcal{C}(\alpha;z)\cdot \big(\Phi_{\alpha,A}(z)+\Phi_{\alpha,B}(z)\big).
\endaligned\end{equation}
Here $\Phi_{\alpha,A}(z)$, $\Phi_{\alpha,B}^{k}(z),k=1,2,3$ are defined in \eqref{phi1}-\eqref{phi2} and
\begin{equation}\aligned\label{phi3}
\mathcal{C}(\alpha;z):=\frac{2}{\pi}y^{\frac{1}{2}}{\alpha}^{-\frac{9}{2}}\big(-\vartheta_{Y}(\frac{y}{\alpha};x)\big)e^{-\pi\alpha y},\:\:\:\:\Phi_{\alpha,B}(z):=\underset{k=1}{\overset{3}{\sum}}\Phi_{\alpha,B}^{k}(z).
\endaligned\end{equation}
\end{lemma}

   In view of the expression of $\mathcal{C}(\alpha;z)$ in Lemma \ref{part1} and the inequality: $-\vartheta_{Y}(X;Y)>0,\:for \:X>0,\:Y\in (0,\frac{1}{2})$ in Lemmas 2.8 and 2.9 of $\mathrm{Luo}$-$\mathrm{Wei}$ \cite{Luo2023}, it is easy to get the following Lemma \ref{l31} which states that the factor $\mathcal{C}(\alpha;z)$ of  $(-\frac{\partial}{\partial x} \mathcal{R}(\alpha;z))$ is positive.
\begin{lemma}\label{l31} Assume $\alpha>0,z\in \mathcal{D}_{\mathcal{G}}$, it holds that $\mathcal{C}(\alpha;z)>0$.
\end{lemma}

Given by Lemmas \ref{part1} and \ref{l31}, to prove Theorem \ref{2Th2}, it is sufficient to prove Proposition \ref{prop}.
\begin{proposition}\label{prop} Assume $\alpha\geq\frac{3}{2},z\in \mathcal{D}_{\mathcal{G}}$,\:it holds that
\begin{equation}\nonumber\aligned
\Phi_{\alpha,A}(z)+\Phi_{\alpha,B}(z)>0.
\endaligned\end{equation}
\end{proposition}

To illustrate Proposition \ref{prop}, we shall prove that $\Phi_{\alpha,A}(z)$ is a positive major part in Lemma \ref{l33} and $\Phi_{\alpha,B}(z)$ is a smaller part comparing $\Phi_{\alpha,A}(z)$ in Lemma \ref{gorgeous1}.

\begin{lemma}\label{l33} Assume $\alpha\geq\frac{3}{2},z\in \mathcal{D}_{\mathcal{G}}$,\:it holds that
\begin{equation}\aligned\label{phii}\nonumber
\Phi_{\alpha,A}(z)>0.
\endaligned\end{equation}
\end{lemma}
Due to the complexity of the proof of Lemma \ref{l33}, we divide it into four cases, which will be demonstrated by Lemmas \ref{l33+1}-\ref{133+4} respectively. For convenience, we denote that:
\begin{equation}\aligned\label{area}
\mathcal{A}_{a}&:=\left\{(\alpha,y) |\:\:y\geq \frac{\sqrt{3}}{2},\alpha\geq\frac{3}{2},y\geq \frac{1}{2}\alpha\right\},\\
\mathcal{A}_{b}&:=\left\{(\alpha,y) |\:\:y\geq 1,2y\leq \alpha\leq 3y \right\},\\
\mathcal{A}_{c}&:=\left\{(\alpha,y) |\:\:y\in[ \frac{\sqrt{3}}{2},1],2y\leq \alpha\leq 3y \right\},\\
\mathcal{A}_{d}&:=\left\{(\alpha,y) |\:\:y\geq \frac{\sqrt{3}}{2},\alpha\geq 3y\right\}.\\
\endaligned\end{equation}
Then
\begin{equation}\aligned\label{con}
\left\{(\alpha,y) |y\geq \frac{\sqrt{3}}{2},\alpha \geq \frac{3}{2}\right\}=\mathcal{A}_{a}\cup \mathcal{A}_{b}\cup \mathcal{A}_{c}\cup \mathcal{A}_{d}.
\endaligned\end{equation}

We deal with the minor part $\Phi_{\alpha,B}(z)$ in Lemma \ref{gorgeous1}, which shows that it is indeed small relative to the principle part $\Phi_{\alpha,A}(z)$.
\begin{lemma}\label{gorgeous1} Assume $\alpha\geq\frac{3}{2},z\in \mathcal{D}_{\mathcal{G}}$,\:it holds that
\begin{equation}\nonumber\aligned
\left|\frac{\Phi_{\alpha,B}(z)}{\Phi_{\alpha,A}(z)}\right|\leq \frac{1}{77},
\endaligned\end{equation}
where $\Phi_{\alpha,A}(z)$ and $\Phi_{\alpha,B}(z)$ are defined in $\eqref{phi1}$-$\eqref{phi3}$.
\end{lemma}

We shall divide the proof of Lemma \ref{gorgeous1} into two cases, namely case one: $\frac{y}{\alpha}\geq \frac{1}{2}$ and case two: $\frac{y}{\alpha}\in( 0,\frac{1}{2})$, which appear respectively in Lemmas \ref{splendid} and \ref{good}. By \eqref{phi3}, Lemmas \ref{l33} and \ref{gorgeous1}, one has
\begin{equation}\nonumber\aligned
\Phi_{\alpha,A}(z)+\Phi_{\alpha,B}(z)\geq \Phi_{\alpha,A}(z)\Big( 1-\left|\frac{\Phi_{\alpha,B}(z)}{\Phi_{\alpha,A}(z)}\right| \Big)\geq \frac{21}{50}\cdot(1-\frac{1}{77})>0.
\endaligned\end{equation}

In conclusion, Lemmas \ref{l33} and \ref{gorgeous1} complete the proof of Proposition \ref{prop}. In the rest of this Section, we shall prove Lemmas \ref{l33} and \ref{gorgeous1}.

\subsection{Region $\mathcal{A}_{a}$: estimate of $\Phi_{\alpha,A}(z)$}\:In this Subsection,\:we shall prove that
\begin{lemma}\label{l33+1} Assume that $(\alpha,y)\in \mathcal{A}_{a}=\left\{(\alpha,y) |\:\:y\geq \frac{\sqrt{3}}{2},\alpha\geq\frac{3}{2},y\geq \frac{1}{2}\alpha\right\}$,\:then for $z\in \mathcal{D}_{\mathcal{G}}$,\:it holds that
\begin{equation}\aligned\nonumber
\Phi_{\alpha,A}(z)\geq D_{1}(\alpha;y)\geq \frac{21}{50},
\endaligned\end{equation}
where
\begin{equation}\aligned\nonumber
D_{1}(\alpha;y):= {\pi}{\alpha}^4 y^2 +{\alpha}^3 y +\frac{3}{4\pi}{\alpha}^2
-(3{\alpha}y  +2\pi{\alpha}^2 y^2 )\cdot\frac{1+\nu(\frac{1}{2})}{1+\mu(\frac{1}{2})}
+{\pi}y^2\cdot \frac{1+\hat{\omega}(\frac{1}{2})}{1+\hat{\mu}(\frac{1}{2})}.
\endaligned\end{equation}
\end{lemma}
\begin{proof}
Since $\frac{y}{\alpha}\geq \frac{1}{2}$, by Lemmas \ref{Lemmatime} and \ref{Lemmavar2}, one has
\begin{equation}\aligned\label{comp}
\Phi_{\alpha,A}(z)
=&{\pi}{\alpha}^4 y^2 +{\alpha}^3 y +\frac{3}{4\pi}{\alpha}^2
+(\frac{3}{\pi}{\alpha}y  +2{\alpha}^2 y^2 )\frac{\vartheta_{XY}(\frac{y}{\alpha};x)}{\vartheta_{Y}(\frac{y}{\alpha};x)}
+ \frac{1}{\pi}y^2 \frac{\vartheta_{XXY}(\frac{y}{\alpha};x)}{\vartheta_{Y}(\frac{y}{\alpha};x)}\\
\geq&{\pi}{\alpha}^4 y^2 +{\alpha}^3 y +\frac{3}{4\pi}{\alpha}^2
-(3{\alpha}y  +2\pi{\alpha}^2 y^2 )\cdot\frac{1+\nu(\frac{y}{\alpha})}{1+\mu(\frac{y}{\alpha})}
+{\pi}y^2\cdot \frac{1+\hat{\omega}(\frac{y}{\alpha})}{1+\hat{\mu}(\frac{y}{\alpha})}, \\
\endaligned\end{equation}
where $\mu(x)$ and $\nu(x)$ are defined in \eqref{duct}. Notice that as $x\geq\frac{1}{2}$,
\begin{equation}\aligned\nonumber
\frac{1+\nu(x)}{1+\mu(x)}\:\: \mathrm{is\: \:decreasing } ,\:\:\frac{1+\hat{\omega}(x)}{1+\hat{\mu}(x)}\: \:\mathrm{is\: \:increasing}.
\endaligned\end{equation}
Then
\begin{equation}\aligned\label{lamp}
\frac{1+\nu(\frac{y}{\alpha})}{1+\mu(\frac{y}{\alpha})}\leq\frac{1+\nu(\frac{1}{2})}{1+\mu(\frac{1}{2})}=1.104299511\cdots,\:\frac{1+\hat{\omega}(\frac{y}{\alpha})}{1+\hat{\mu}(\frac{y}{\alpha})}\geq \frac{1+\hat{\omega}(\frac{1}{2})}{1+\hat{\mu}(\frac{1}{2})}=0.4435351039\cdots.
\endaligned\end{equation}
By \eqref{comp}-\eqref{lamp}, one has
\begin{equation}\aligned\nonumber
\Phi_{\alpha,A}(z)\geq D_{1}(\alpha;y). \\
\endaligned\end{equation}
A direct checking shows that
\begin{equation}\aligned\label{middleage1}
\frac{\partial}{\partial{\alpha}}D_{1}(\alpha;y)>0,\: \frac{\partial}{\partial{y}}D_{1}(\alpha;y)>0,\:\: as \:\: \alpha\geq \frac{3}{2}, \:y\geq \frac{\sqrt{3}}{2}.
\endaligned\end{equation}
By \eqref{middleage1}, one deduces that
\begin{equation}\aligned\nonumber
D_{1}(\alpha;y)\geq D_{1}(\frac{3}{2};\frac{\sqrt{3}}{2})\geq \frac{21}{50}.
\endaligned\end{equation}
\end{proof}

\subsection{Region $\mathcal{A}_{b}$: estimate of $\Phi_{\alpha,A}(z)$} In this Subsection, we shall prove that
\begin{lemma}\label{l33+2} Assume that $(\alpha,y)\in \mathcal{A}_{b}=\left\{(\alpha,y) |\:\:y\geq 1,2y\leq \alpha\leq 3y \right\}$,\:then for $z\in \mathcal{D}_{\mathcal{G}}$,\:it holds that
\begin{equation}\aligned\nonumber
\Phi_{\alpha,A}(z)\geq D_{2}(\alpha;y)\geq\frac{39}{5},
\endaligned\end{equation}
where
\begin{equation}\aligned\nonumber
D_{2}(\alpha;y):={\pi}{\alpha}^4 y^2 +{\alpha}^3 y +\frac{3}{4\pi}{\alpha}^2
-(3{\alpha}y  +2\pi{\alpha}^2 y^2 )\cdot\frac{1+\nu(\frac{1}{3})}{1+\mu(\frac{1}{3})}
+{\pi}y^2\cdot \frac{1+\hat{\omega}(\frac{1}{3})}{1+\hat{\mu}(\frac{1}{3})}.
\endaligned\end{equation}
\end{lemma}
\begin{proof}
Since $\frac{y}{\alpha}\in [\frac{1}{3},\frac{1}{2}]$, by Lemmas \ref{Lemmatime} and \ref{Lemmavar2}, one has
\begin{equation}\aligned\label{comppj}
\Phi_{\alpha,A}(z)
=&{\pi}{\alpha}^4 y^2 +{\alpha}^3 y +\frac{3}{4\pi}{\alpha}^2
+(\frac{3}{\pi}{\alpha}y  +2{\alpha}^2 y^2 )\frac{\vartheta_{XY}(\frac{y}{\alpha};x)}{\vartheta_{Y}(\frac{y}{\alpha};x)}
+ \frac{1}{\pi}y^2 \frac{\vartheta_{XXY}(\frac{y}{\alpha};x)}{\vartheta_{Y}(\frac{y}{\alpha};x)}\\
\geq &{\pi}{\alpha}^4 y^2 +{\alpha}^3 y +\frac{3}{4\pi}{\alpha}^2
-(3{\alpha}y  +2\pi{\alpha}^2 y^2 )\cdot\frac{1+\nu(\frac{y}{\alpha})}{1+\mu(\frac{y}{\alpha})}
+{\pi}y^2\cdot \frac{1+\hat{\omega}(\frac{y}{\alpha})}{1+\hat{\mu}(\frac{y}{\alpha})},\\
\endaligned\end{equation}
where $\mu(x)$ and $\nu(x)$ are defined in \eqref{duct}. Notice that as $x\geq\frac{1}{3}$,
\begin{equation}\aligned\nonumber
\frac{1+\nu(x)}{1+\mu(x)}\:\: \mathrm{is\:\: decreasing } ,\:\:\frac{1+\hat{\omega}(x)}{1+\hat{\mu}(x)} \:\:\mathrm{is \:\:increasing}.
\endaligned\end{equation}
Then
\begin{equation}\aligned\label{lampp}
\frac{1+\nu(\frac{y}{\alpha})}{1+\mu(\frac{y}{\alpha})}\leq\frac{1+\nu(\frac{1}{3})}{1+\mu(\frac{1}{3})}=1.455483937\cdots,\:\frac{1+\hat{\omega}(\frac{y}{\alpha})}{1+\hat{\mu}(\frac{y}{\alpha})}\geq \frac{1+\hat{\omega}(\frac{1}{3})}{1+\hat{\mu}(\frac{1}{3})}=-1.927931130\cdots.
\endaligned\end{equation}
Combining \eqref{comppj}-\eqref{lampp} with the condition $\alpha \geq 2,\:y\geq 1$, one has
\begin{equation}\aligned\nonumber
\Phi_{\alpha,A}(z)\geq D_{2}(\alpha;y)\geq D_{2}(2;1)\geq \frac{39}{5}.\\
\endaligned\end{equation}
\end{proof}

\subsection{Region $\mathcal{A}_{c}$: estimate of $\Phi_{\alpha,A}(z)$} In this Subsection, we shall prove that
\begin{lemma}\label{l33+3} Assume that $(\alpha,y)\in \mathcal{A}_{c}=\left\{(\alpha,y) |\:\:y\in[ \frac{\sqrt{3}}{2},1],2y\leq \alpha\leq 3y \right\}$,\:then for $z\in \mathcal{D}_{\mathcal{G}}$,\:it holds that
\begin{equation}\aligned\nonumber
\Phi_{\alpha,A}(z)\geq D_{3}(\alpha;y)\geq \frac{109}{10},
\endaligned\end{equation}
where
\begin{equation}\label{lan1} \aligned
D_{3}(\alpha;y):=&{\pi}{\alpha}^4 y^2 +{\alpha}^3 y +\frac{3}{4\pi}{\alpha}^2
+(\frac{3}{\pi}{\alpha}y  +2{\alpha}^2 y^2 )\frac{\vartheta_{XY}(\frac{y}{\alpha};\sqrt{1-{y}^2})}{\vartheta_{Y}(\frac{y}{\alpha};\sqrt{1-{y}^2})}
+\frac{1}{\pi}y^2 \frac{\vartheta_{XXY}(\frac{y}{\alpha};\frac{1}{2})}{\vartheta_{Y}(\frac{y}{\alpha};\frac{1}{2})}.\\
\endaligned\end{equation}
\end{lemma}
\begin{proof}
In view of  Lemmas \ref{mono} and \ref{monoo}, one has
\begin{equation}\label{explain} \aligned
\frac{\partial{}}{\partial{Y}}  \frac{\vartheta_{XY}(X;Y)}{\vartheta_{Y}(X;Y)} \geq 0, \:\frac{\partial{}}{\partial{Y}} \frac{\vartheta_{XXY}(X;Y)}{\vartheta_{Y}(X;Y)} \leq 0,\:\:for\:\:X\geq\frac{1}{3},\:0\leq Y\leq\frac{1}{2}.
\endaligned\end{equation}
Note that the region that point $z$  belongs to is
\begin{equation}\label{account} \aligned\nonumber
z\in\{(x,y)|\:0<x<\frac{1}{2},\:x^2+y^2>1,\:\frac{\sqrt{3}}{2}<y<1   \}.
\endaligned\end{equation}
 Combining \eqref{explain} with $\sqrt{1-{y}^2}<x<\frac{1}{2}$, one has
\begin{equation} \aligned\nonumber
\Phi_{\alpha,A}(z)=&{\pi}{\alpha}^4 y^2 +{\alpha}^3 y +\frac{3}{4\pi}{\alpha}^2
+(\frac{3}{\pi}{\alpha}y  +2{\alpha}^2 y^2 )\frac{\vartheta_{XY}(\frac{y}{\alpha};x)}{\vartheta_{Y}(\frac{y}{\alpha};x)}
+\frac{1}{\pi}y^2 \frac{\vartheta_{XXY}(\frac{y}{\alpha};x)}{\vartheta_{Y}(\frac{y}{\alpha};x)}\\
\geq& D_{3}(\alpha;y).
\endaligned\end{equation}
 Considering the condition that $(\alpha,y)\in \mathcal{A}_{c}$, one gets that $\alpha\in[\sqrt{3},3],y\in[\frac{\sqrt{3}}{2},1]$. And $D_{3}(\alpha;y)$ is increasing  with respect to $\alpha$  in this case actually (see Figure \ref{Phi}).
Then,
\begin{equation}\label{lan} \aligned
\Phi_{\alpha,A}(z)&\geq D_{3}(\alpha;y)\geq D_{3}(\sqrt{3};y).
\endaligned\end{equation}
Consequently, \eqref{lan} and Lemma \ref{Lemmafuture} lead to the result.
\end{proof}

\begin{figure}
\centering
 \includegraphics[scale=0.45]{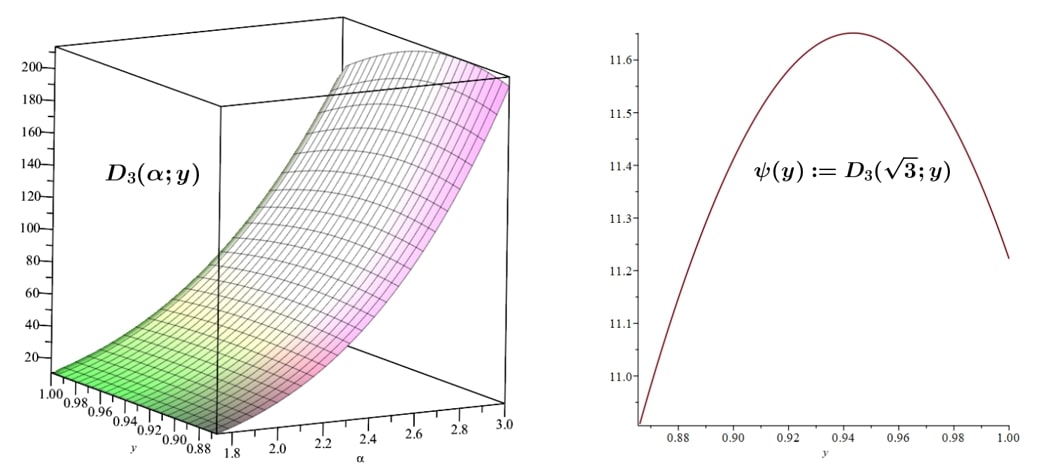}
 \caption{The images of $D_{3}(\alpha;y)$ and\:$\psi(y)$.}
 \label{Phi}
\end{figure}

\begin{lemma}\label{Lemmafuture}
Let $\psi(y):=D_{3}(\sqrt{3};y)$, where $D_{3}(\alpha;y)$ is defined in \eqref{lan1}. Then
 \begin{equation}\nonumber\aligned
\underset{y\in[\frac{\sqrt{3}}{2},1]}{\min}\psi(y)=\min\{ \psi(\frac{\sqrt{3}}{2}),\psi(1) \}=\psi(\frac{\sqrt{3}}{2})\geq\frac{109}{10}.
\endaligned\end{equation}
\end{lemma}
$ \psi(y)$ is a concave function as $y$ is in the interval $[\frac{\sqrt{3}}{2},1]$ (see Figure \ref{Phi}). Since $\psi(y)$ is a function of one variable on a short interval, the detail for proving that\: $\psi''(y) \leq 0 $ is omitted here. Next we shall compute $\psi(\frac{\sqrt{3}}{2})$.
By \eqref{nice}, \eqref{highh1} and \eqref{lan1}, one has
\begin{equation}\aligned\nonumber
\psi(\frac{\sqrt{3}}{2})=\frac{27}{4}\pi +\frac{9}{2} +\frac{9}{4\pi}
+\frac{9+9\pi}{2\pi}\cdot\frac{\vartheta_{XY}(\frac{1}{2};\frac{1}{2})}{\vartheta_{Y}(\frac{1}{2};\frac{1}{2})}
+\frac{3}{4\pi}\frac{\vartheta_{XXY}(\frac{1}{2};\frac{1}{2})}{\vartheta_{Y}(\frac{1}{2};\frac{1}{2})}=10.90887470\cdots.
\endaligned\end{equation}

\subsection{Region $\mathcal{A}_{d}$: estimate of $\Phi_{\alpha,A}(z)$}\:In this Subsection,\:we shall prove that
\begin{lemma}\label{133+4} Assume that $(\alpha,y)\in \mathcal{A}_{d}=\left\{(\alpha,y) |\:\:y\geq \frac{\sqrt{3}}{2},\alpha\geq 3y\right\}$, then for $z\in \mathcal{D}_{\mathcal{G}}$, it holds that
 \begin{equation}\nonumber\aligned
\Phi_{\alpha,A}(z)\geq D_{4}(\alpha;y)\geq 38,
\endaligned\end{equation}
where
\begin{equation}\aligned\nonumber
D_{4}(\alpha;y):={\pi}{\alpha}^4 y^2 +{\alpha}^3 y +\frac{3}{4\pi}{\alpha}^2
-(\frac{3}{\pi}\alpha{y}+2{\alpha}^2y^2)\cdot\frac{\frac{3y^2}{4{\alpha}^2}+2{\pi}^2e^{-\frac{\pi\alpha}{y}}}{\frac{y^3}{2{\alpha}^3}-2\pi\frac{y^2}{{\alpha}^2}e^{-\frac{\pi\alpha}{y}}}
+\frac{15}{4\pi}{\alpha}^2(1-\frac{\pi\alpha}{3y}-\frac{{\pi}^2{\alpha}^2}{60 {y}^2}).\\
\endaligned\end{equation}
\end{lemma}
\begin{proof}
Since $\frac{y}{\alpha}\in (0,\frac{1}{3}]$, by Lemmas \ref{Lemmatime} and \ref{Lemmavar2}, one has
\begin{equation}\aligned\nonumber
\Phi_{\alpha,A}(z)=&{\pi}{\alpha}^4 y^2 +{\alpha}^3 y +\frac{3}{4\pi}{\alpha}^2
+(\frac{3}{\pi}{\alpha}y  +2{\alpha}^2 y^2 )\frac{\vartheta_{XY}(\frac{y}{\alpha};x)}{\vartheta_{Y}(\frac{y}{\alpha};x)}
+ \frac{1}{\pi}y^2 \frac{\vartheta_{XXY}(\frac{y}{\alpha};x)}{\vartheta_{Y}(\frac{y}{\alpha};x)}\\
\geq &D_{4}(\alpha;y).
\endaligned\end{equation}
Notice that $(\alpha,y)\in \mathcal{A}_{d}$, then $\alpha\geq \frac{3\sqrt{3}}{2},y\geq \frac{\sqrt{3}}{2}$. A direct checking shows that
\begin{equation}\aligned\nonumber
 D_{4}(\alpha;y)\geq D_{4}(\frac{3\sqrt{3}}{2};\frac{\sqrt{3}}{2})\geq 38.
\endaligned\end{equation}
\end{proof}

Through summarizing the estimates of $\Phi_{\alpha,A}(z)$ in the domains $\mathcal{A}_{a},\mathcal{A}_{b},\mathcal{A}_{c},\mathcal{A}_{d}$, we obtain the proof of Lemma \ref{l33}. To prove Lemma \ref{gorgeous1}, in the following two Lemmas, we shall estimate $\left|\frac{\Phi_{\alpha,B}(z)}{\Phi_{\alpha,A}(z)}\right|$ to indicate that $\Phi_{\alpha,B}(z)$ is small compared to $\Phi_{\alpha,A}(z)$.

\subsection{Region $\mathcal{A}_{a}$: estimate of $\left|\frac{\Phi_{\alpha,B}(z)}{\Phi_{\alpha,A}(z)}\right|$} In this Subsection, we shall prove that
\begin{lemma}\label{splendid} Assume that $(\alpha,y)\in \mathcal{A}_{a}=\left\{(\alpha,y) |\:\:y\geq \frac{\sqrt{3}}{2},\alpha\geq\frac{3}{2},y\geq \frac{1}{2}\alpha\right\}$, then for $z\in \mathcal{D}_{\mathcal{G}}$, it holds that
\begin{equation}\nonumber\aligned
\left|\frac{\Phi_{\alpha,B}(z)}{\Phi_{\alpha,A}(z)}\right|\leq\frac{1}{77}.
\endaligned\end{equation}
where $\Phi_{\alpha,A}(z)$ and $\Phi_{\alpha,B}(z)$ are defined in \eqref{phi1}-\eqref{phi3}.
\end{lemma}
\begin{proof}
By Lemma \ref{l33+1}, one has $\Phi_{\alpha,A}(z)\geq \frac{21}{50}$. And by \eqref{phi3}, one has $\left|{\Phi_{\alpha,B}(z)}\right|\leq \underset{i=1}{\overset{3}{\sum}}\left|\Phi_{\alpha,B}^{i}(z)\right|$. Then, it suffices to prove that
\begin{equation}\aligned\nonumber
\frac{50}{21}\:\underset{i=1}{\overset{3}{\sum}}\left|\Phi_{\alpha,B}^{i}(z)\right|\leq\frac{1}{77}.
\endaligned\end{equation}
Before proceeding, we need to explain some inequalities to be utilized. Notice that $\mu(x),\nu(x),\omega(x)$ defined in \eqref{duct} are decreasing as $x\geq \frac{1}{2}$. Then one has
\begin{equation}\aligned\label {front11}
\frac{1+\mu(x)}{1-\mu(x)}\leq \frac{1+\mu(\frac{1}{2})}{1-\mu(\frac{1}{2})},\:\:\frac{1+\nu(x)}{1-\mu(x)}\leq \frac{1+\nu(\frac{1}{2})}{1-\mu(\frac{1}{2})},\:\:\frac{1+\omega(x)}{1-\mu(x)}\leq \frac{1+\omega(\frac{1}{2})}{1-\mu(\frac{1}{2})}.
\endaligned\end{equation}
To estimate $\Phi_{\alpha,B}^{1}(z),\Phi_{\alpha,B}^{2}(z)$ and $\Phi_{\alpha,B}^{3}(z)$, we denote that
\begin{equation}\aligned\nonumber
E_{1}(\alpha;y):=&\frac{1+\mu(\frac{1}{2})}{1-\mu(\frac{1}{2})}\cdot\sum_{n=2}^\infty ({\pi}{\alpha}^4 y^2 n^6+{\alpha}^3 y n^4+\frac{3}{4\pi}{\alpha}^2 n^2)e^{-\alpha\pi y(n^2-1)}.\\
E_{2}(\alpha;y):=&\frac{1+\nu(\frac{1}{2})}{1-\mu(\frac{1}{2})}\cdot \sum_{n=2}^\infty (3{\alpha}y  n^2 +2\pi{\alpha}^2 y^2 n^4)e^{-\alpha\pi y(n^2-1)}.\\
E_{3}(\alpha;y):=&\frac{1+\omega(\frac{1}{2})}{1-\mu(\frac{1}{2})}\cdot\sum_{n=2}^\infty \pi{y}^2  n^2 e^{-\alpha\pi y(n^2-1)}.
\endaligned\end{equation}
Firstly, about $\Phi_{\alpha,B}^{1}(z)$, by Lemma \ref{Lemmamoist}, inequality \eqref{front11} and the condition $\alpha\geq \frac{3}{2},y\geq \frac{\sqrt{3}}{2}$, one has
\begin{equation}\aligned\label {middle1}
\left|\Phi_{\alpha,B}^{1}(z)\right|=&\sum_{n=2}^\infty ({\pi}{\alpha}^4 y^2 n^5+{\alpha}^3 y n^3+\frac{3}{4\pi}{\alpha}^2 n)e^{-\alpha\pi y(n^2-1)}\cdot\left|\frac{\vartheta_{Y}(\frac{y}{\alpha};nx)}{\vartheta_{Y}(\frac{y}{\alpha};x)}\right|\\
\leq& \frac{1+\mu(\frac{y}{\alpha})}{1-\mu(\frac{y}{\alpha})} \cdot\sum_{n=2}^\infty ({\pi}{\alpha}^4 y^2 n^6+{\alpha}^3 y n^4+\frac{3}{4\pi}{\alpha}^2 n^2)e^{-\alpha\pi y(n^2-1)}\\
\leq& E_{1}(\alpha;y)\leq E_{1}(\frac{3}{2};\frac{\sqrt{3}}{2})\leq \frac{1}{235}.
\endaligned\end{equation}
Secondly, about $\Phi_{\alpha,B}^{2}(z)$, combining Lemma \ref{Lemmamoistt}, inequality \eqref{front11} and $\alpha\geq \frac{3}{2},y\geq \frac{\sqrt{3}}{2}$, one has
\begin{equation}\aligned\label {middle4}
\left|\Phi_{\alpha,B}^{2}(z)\right|=&\sum_{n=2}^\infty (\frac{3}{\pi}{\alpha}y  n+2{\alpha}^2 y^2 n^3)e^{-\alpha\pi y(n^2-1)}\cdot\left|\frac{\vartheta_{XY}(\frac{y}{\alpha};nx)}{\vartheta_{Y}(\frac{y}{\alpha};x)}\right|\\
\leq& \frac{1+\nu(\frac{y}{\alpha})}{1-\mu(\frac{y}{\alpha})} \cdot \sum_{n=2}^\infty (3{\alpha}y  n^2 +2\pi{\alpha}^2 y^2 n^4)e^{-\alpha\pi y(n^2-1)}\\
\leq& E_{2}(\alpha;y)\leq E_{2}(\frac{3}{2};\frac{\sqrt{3}}{2})\leq \frac{1}{920}.
\endaligned\end{equation}
Thirdly, about $\Phi_{\alpha,B}^{3}(z)$, combining Lemma \ref{Lemmavar1}, inequality \eqref{front11} and $\alpha\geq \frac{3}{2},y\geq \frac{\sqrt{3}}{2}$, one has
\begin{equation}\aligned\label {middle5}
\left|\Phi_{\alpha,B}^{3}(z)\right|=&\sum_{n=2}^\infty \frac{1}{\pi}y^2 ne^{-\alpha\pi y(n^2-1)}\cdot\left|\frac{\vartheta_{XXY}(\frac{y}{\alpha};nx)}{\vartheta_{Y}(\frac{y}{\alpha};x)}\right|\\
\leq& \frac{1+\omega(\frac{y}{\alpha})}{1-\mu(\frac{y}{\alpha})} \cdot \sum_{n=2}^\infty \pi{y}^2  n^2 e^{-\alpha\pi y(n^2-1)}\\
\leq& E_{3}(\alpha;y)\leq E_{3}(\frac{3}{2};\frac{\sqrt{3}}{2})\leq  10^{-4}.
\endaligned\end{equation}
Consequently, \eqref{middle1}-\eqref{middle5} yield the result.
\end{proof}

\subsection{Region $\mathcal{A}_{b}\cup\mathcal{A}_{c}\cup\mathcal{A}_{d}$: estimate of $\left|\frac{\Phi_{\alpha,B}(z)}{\Phi_{\alpha,A}(z)}\right|$} In this Subsection, we shall prove that
\begin{lemma}\label{good} Assume that $(\alpha,y)\in \mathcal{A}_{b}\cup\mathcal{A}_{c}\cup\mathcal{A}_{d}=\left\{(\alpha,y) |\:\:y\geq \frac{\sqrt{3}}{2},\alpha\geq 2 y\right\}$, then for $z\in \mathcal{D}_{\mathcal{G}}$, it holds that
\begin{equation}\nonumber\aligned
\left|\frac{\Phi_{\alpha,B}(z)}{\Phi_{\alpha,A}(z)}\right|\leq 10^{-3}.
\endaligned\end{equation}
where $\Phi_{\alpha,A}(z)$ and $\Phi_{\alpha,B}(z)$ are defined in $\eqref{phi1}$-$\eqref{phi2}$.
\end{lemma}
\begin{proof}
Combining Lemmas \ref{l33+2}-\ref{133+4}, under the hypothetical conditions one concludes that:
\begin{equation}\aligned\nonumber
\Phi_{\alpha,A}(z)\geq \frac{39}{5}.
\endaligned\end{equation}
By \eqref{phi3}, one has $\left|{\Phi_{\alpha,B}(z)}\right|\leq \underset{i=1}{\overset{3}{\sum}}\left|\Phi_{\alpha,B}^{i}(z)\right|.$ Therefore it suffices to prove that
\begin{equation}\aligned\nonumber
\frac{5}{39}\underset{i=1}{\overset{3}{\sum}}\left|\Phi_{\alpha,B}^{i}(z)\right|\leq 10^{-3}.
\endaligned\end{equation}
Before proceeding, it is necessary to notice the range of $(\alpha;y)$. Considering that $\frac{y}{\alpha}\in (0,\frac{1}{2})$ and $z\in \mathcal{D}_{\mathcal{G}}$, one has $(\alpha,y)\in(\sqrt{3},\infty)\times(\frac{\sqrt{3}}{2},\infty) $. And we denote that
\begin{equation}\aligned\label {middlee1}\nonumber
\tilde{E}_{1}(\alpha;y):=&\sum_{n=2}^\infty ({\alpha}^4 y^2 n^6+\frac{1}{\pi}{\alpha}^3 y n^4+\frac{3}{4{\pi}^2}{\alpha}^2 n^2)e^{-\alpha\pi [y(n^2-1)-\frac{1}{4y}]},\\
\tilde{E}_{2}(\alpha;y):=&\sum_{n=2}^\infty (\frac{9}{2{\pi}^2}{\alpha}^2{n}^2+\frac{3}{\pi}{\alpha}^3{y}{n}^4)(1+\frac{\pi\alpha}{6y})e^{-\alpha\pi [y(n^2-1)-\frac{1}{4y}]},\\
\tilde{E}_{3}(\alpha;y):=&\sum_{n=2}^\infty \frac{15}{4{\pi}^2} {\alpha}^2{n}^2(1+\frac{\pi\alpha}{3y}+\frac{{\pi}^2}{60}\frac{{\alpha}^2}{y^2})  e^{-\alpha\pi [y(n^2-1)-\frac{1}{4y}]}.
\endaligned\end{equation}
Firstly, about $\Phi_{\alpha,B}^{1}(z)$, using Lemma \ref{Lemmamoist}, one has
\begin{equation}\aligned\nonumber
\left|\Phi_{\alpha,B}^{1}(z)\right|=&\sum_{n=2}^\infty ({\pi}{\alpha}^4 y^2 n^5+{\alpha}^3 y n^3+\frac{3}{4\pi}{\alpha}^2 n)e^{-\alpha\pi y(n^2-1)}\cdot\left|\frac{\vartheta_{Y}(\frac{y}{\alpha};nx)}{\vartheta_{Y}(\frac{y}{\alpha};x)}\right|\\
\leq& \tilde{E}_{1}(\alpha;y)\leq \tilde{E}_{1}(\sqrt{3};\frac{\sqrt{3}}{2})\leq 2\cdot 10^{-3}.
\endaligned\end{equation}
Secondly, about $\Phi_{\alpha,B}^{2}(z)$, using Lemma \ref{Lemmamoistt}, one has
\begin{equation}\aligned\nonumber
\left|\Phi_{\alpha,B}^{2}(z)\right|=&\sum_{n=2}^\infty (\frac{3}{\pi}{\alpha}y  n+2{\alpha}^2 y^2 n^3)e^{-\alpha\pi y(n^2-1)}\cdot\left|\frac{\vartheta_{XY}(\frac{y}{\alpha};nx)}{\vartheta_{Y}(\frac{y}{\alpha};x)}\right|\\
\leq& \tilde{E}_{2}(\alpha;y)\leq\tilde{E}_{2}(\sqrt{3};\frac{\sqrt{3}}{2})\leq 6\cdot 10^{-4}.
\endaligned\end{equation}
Thirdly, about $\Phi_{\alpha,B}^{3}(z)$, by Lemma \ref{Lemmavar1}, one has
\begin{equation}\aligned\nonumber
\left|\Phi_{\alpha,B}^{3}(z)\right|=&\sum_{n=2}^\infty \frac{1}{\pi}y^2 ne^{-\alpha\pi y(n^2-1)}\cdot\left|\frac{\vartheta_{XXY}(\frac{y}{\alpha};nx)}{\vartheta_{Y}(\frac{y}{\alpha};x)}\right|\\
\leq& \tilde{E}_{3}(\alpha;y)\leq\tilde{E}_{3}(\sqrt{3};\frac{\sqrt{3}}{2})\leq 6\cdot 10^{-5}.
\endaligned\end{equation}
\end{proof}

\section{The monotonicity on the vertical line $y=\frac{1}{2}$}
\setcounter{equation}{0}
In Theorem \ref{2Th1}, we have established that for $\alpha\geq \frac{3}{2}$,
\begin{equation}\aligned\label{3hhh}
\min_{z\in\mathbb{H}}\mathcal{R}(\alpha;z)
=\min_{z\in\Gamma}\mathcal{R}(\alpha;z)
\endaligned\end{equation}
where $\Gamma$ is a vertical line and defined in \eqref{boundary}.
In this Section, we aim to prove Theorem \ref{4Th1}.

\begin{theorem}\label{4Th1} Assume that $\alpha\geq \frac{3}{2}$, it holds that
\begin{equation}\aligned\nonumber
\min_{z\in\Gamma}\mathcal{R}(\alpha;z) \:\:is\:\: achieved \:\: uniquely\:\:at \:\:\frac{1}{2}+i\frac{\sqrt{3}}{2}.
\endaligned\end{equation}
\end{theorem}

Note that Theorems \ref{2Th1} and \ref{4Th1}  comprise the proof of Theorem \ref{Th2}, which is our main result in this paper.  And the proof of Theorem \ref{4Th1} is based on the following Proposition.

\begin{proposition}\label{restart} Assume that $\alpha\geq \frac{3}{2} ,\:y\geq \frac{\sqrt{3}}{2}$, then
\begin{equation}\aligned\nonumber
\frac{\partial}{\partial{y}}\mathcal{R}(\alpha;\frac{1}{2}+iy)\geq 0.
\endaligned\end{equation}
\end{proposition}

To prove Proposition \ref{restart}, we need Lemma \ref{4lematem} deduced by Proposition 3.4 of B\'etermin \cite{Bet2018}.

\begin{lemma}[B\'etermin \cite{Bet2018}]\label{4lematem}
Assume that $\alpha \geq \frac{3}{2}$, it holds that $\frac{\partial}{\partial y}\mathcal{R}(\alpha;\frac{1}{2}+i \frac{\sqrt{3}}{2})=0$.
\end{lemma}

  By Lemma \ref{4lematem}, we know that $\frac{\partial}{\partial{y}}\mathcal{R}(\alpha;\frac{1}{2}+iy)$ approaches to $0$ as $y$ approaches to $\frac{\sqrt{3}}{2}$. Therefore, we divide the proof of Proposition \ref{restart} into two cases, one where $y$ is away from $\frac{\sqrt{3}}{2}$, and the other more complicated case where $y$ is close to $\frac{\sqrt{3}}{2}$. To make it clear, we denote that
\begin{equation}\aligned\label{area2}
\Omega_{1}&:=\left\{(\alpha,y) |\:\:\alpha\geq\frac{3}{2},y\geq \frac{4}{5}\alpha\right\},\\
\Omega_{2}&:=\left\{(\alpha,y) |\:\:\alpha\geq\frac{3}{2},y\in[\frac{\sqrt{3}}{2},\frac{4}{5}\alpha]\right\}.\\
\endaligned\end{equation}
Then
\begin{equation}\aligned\label{con}
\left\{(\alpha,y) |\alpha \geq \frac{3}{2},y\geq \frac{\sqrt{3}}{2}\right\}=\Omega_{1}\cup \Omega_{2}.
\endaligned\end{equation}

In $\Omega_{1}$ and $\Omega_{2}$, we take advantage of different methods. In $\Omega_{1}$, we directly estimate $\frac{\partial}{\partial{y}}\mathcal{R}(\alpha;\frac{1}{2}+iy)$ by exponential expansion. In $\Omega_{2}$, combining the equality $\frac{\partial^2}{\partial y^2}+\frac{2}{y}\frac{\partial}{\partial y}=\frac{1}{y^2}\frac{\partial}{\partial y}\big(y^2 \frac{\partial}{\partial y} \big)$  with $\big(y^2 \frac{\partial}{\partial y} \big)\mathcal{R}(\alpha;\frac{1}{2}+iy)|_{y=\frac{\sqrt{3}}{2}}=0$ derived from Lemma \ref{4lematem}, in fact, we only need to estimate that $(\frac{\partial^2}{\partial y^2}+\frac{2}{y}\frac{\partial}{\partial y})\mathcal{R}(\alpha;\frac{1}{2}+iy)$ is nonnegative, which yields that $\frac{\partial}{\partial{y}}\mathcal{R}(\alpha;\frac{1}{2}+iy)$ is nonnegative.

\subsection{Region $\Omega_{1}$: estimate of $\frac{\partial}{\partial y}\mathcal{R}(\alpha;\frac{1}{2}+iy)$}

To better elucidate the estimate of $\frac{\partial}{\partial y}\mathcal{R}(\alpha;\frac{1}{2}+iy)$, we denote that
\begin{equation}\aligned\label{meani}
I_{a,1}(\alpha;y)&:=\underset{\left|n\right|\leq 1}{\sum}(-{\pi}^3{\alpha}^5{y}^3{n}^6
+\frac{3}{2}{\pi}^2{\alpha}^4{y}^2{n}^4+\frac{3}{4}\pi{\alpha}^3y{n}^2+\frac{3}{8}{\alpha}^2)e^{-{\pi}{\alpha}y{n}^2 }\vartheta(\frac{y}{\alpha};\frac{n}{2}),\\
I_{a,2}(\alpha;y)&:=\underset{\left|n\right|\leq 1}{\sum}(-{\pi}^2{\alpha}^3{y}^3{n}^4
+3\pi{\alpha}^2{y}^2{n}^2+\frac{21}{4}{\alpha}y)e^{-{\pi}{\alpha}y{n}^2 }\vartheta_{X}(\frac{y}{\alpha};\frac{n}{2}),\\
I_{a,3}(\alpha;y)&:=\underset{\left|n\right|\leq 1}{\sum}({\pi}{\alpha}{y}^3{n}^2
+\frac{11}{2}{y}^2)e^{-{\pi}{\alpha}y{n}^2 }\vartheta_{XX}(\frac{y}{\alpha};\frac{n}{2}),\\
I_{a,4}(\alpha;y)&:=\underset{\left|n\right|\leq 1}{\sum}{\alpha}^{-1}{y}^3e^{-{\pi}{\alpha}y{n}^2 }\vartheta_{XXX}(\frac{y}{\alpha};\frac{n}{2}).
\endaligned\end{equation}
and
\begin{equation}\aligned\label{mean00}
I_{b,1}(\alpha;y)&:=\underset{\left|n\right|\geq 2}{\sum}(-{\pi}^3{\alpha}^5{y}^3{n}^6
+\frac{3}{2}{\pi}^2{\alpha}^4{y}^2{n}^4+\frac{3}{4}\pi{\alpha}^3y{n}^2+\frac{3}{8}{\alpha}^2)e^{-{\pi}{\alpha}y{n}^2 }\vartheta(\frac{y}{\alpha};\frac{n}{2}),\\
I_{b,2}(\alpha;y)&:=\underset{\left|n\right|\geq 2}{\sum}(-{\pi}^2{\alpha}^3{y}^3{n}^4
+3\pi{\alpha}^2{y}^2{n}^2+\frac{21}{4}{\alpha}y)e^{-{\pi}{\alpha}y{n}^2 }\vartheta_{X}(\frac{y}{\alpha};\frac{n}{2}),\\
I_{b,3}(\alpha;y)&:=\underset{\left|n\right|\geq 2}{\sum}({\pi}{\alpha}{y}^3{n}^2
+\frac{11}{2}{y}^2)e^{-{\pi}{\alpha}y{n}^2 }\vartheta_{XX}(\frac{y}{\alpha};\frac{n}{2}),\\
I_{b,4}(\alpha;y)&:=\underset{\left|n\right|\geq 2}{\sum}{\alpha}^{-1}{y}^3e^{-{\pi}{\alpha}y{n}^2 }\vartheta_{XXX}(\frac{y}{\alpha};\frac{n}{2}).
\endaligned\end{equation}

Based on Lemmas \ref{respect} and \ref{Lemmatree}, we first establish the following expression by \eqref{meani}-\eqref{mean00}. The computation is straightforward, so we omit it.

\begin{lemma}\label{estab}$(\mathrm{An\: identity \:for}\:\frac{\partial}{\partial{y}}\mathcal{R}(\alpha;\frac{1}{2}+iy))$\: For $\alpha,\:y>0$,\:it holds that
  \begin{equation}\aligned\label{regroup}
\frac{\partial}{\partial{y}}\mathcal{R}(\alpha;\frac{1}{2}+iy)=\frac{1}{{\pi}^2}{\alpha}^{-\frac{9}{2}}{y}^{-\frac{1}{2}}\big(I_{a}(\alpha;y)+I_{b}(\alpha;y)\big),
\endaligned\end{equation}
where
\begin{equation}\aligned\label{mean0}
I_{a}(\alpha;y):=\sum_{k=1}^{4}I_{a,k}(\alpha;y),\:\:\:\:I_{b}(\alpha;y):=\sum_{k=1}^{4}I_{b,k}(\alpha;y).
\endaligned\end{equation}
\end{lemma}

In this Subsection, we aim to prove that

\begin{lemma}\label{mean5} Assume that $(\alpha,y)\in \Omega_{1}$, then
  \begin{equation}\aligned\nonumber
\frac{\partial}{\partial{y}}\mathcal{R}(\alpha;\frac{1}{2}+iy) >0.
\endaligned\end{equation}
\end{lemma}

To prove Lemma \ref{mean5}, by Lemma \ref{estab}, we aim to demonstrate that $I_{a}(\alpha;y)$ is a positive principle part in Lemma \ref{maea3} and $I_{b}(\alpha;y)$ is an error part which is very small relative to the principle part in Lemma \ref{mean4}. In other words, the sigh of $\frac{\partial}{\partial{y}}\mathcal{R}(\alpha;\frac{1}{2}+iy)$ is determined by $I_{a}(\alpha;y)$ as $(\alpha,y)\in \Omega_{1}$.

\begin{lemma}[A lower bound estimate of $I_{a}(\alpha;y)$ in \eqref{mean0} and \eqref{meani}]\label{maea3} Assume that $(\alpha,y)\in \Omega_{1}$, then
  \begin{equation}\aligned\nonumber
I_{a}(\alpha;y)\geq\frac{3}{40}{\alpha}^2.
\endaligned\end{equation}
\end{lemma}
\begin{proof}
For convenience, we denote that
\begin{equation}\aligned\label{wen}
 \mathcal{P}(\alpha;y):=&1+2 e^{-\pi\frac{y}{\alpha}}\big( 1+\frac{44}{3}{\pi}^2{\alpha}^{-2}{y}^{2} -14\pi {\alpha}^{-1}y\cdot(1+\overline{\epsilon}_{2})-\frac{8}{3}{\pi}^3{\alpha}^{-3}y^3\cdot(1+ \overline{\epsilon}_{3})\big)\\
+&\frac{8}{3}e^{-\pi{y}(\alpha+\frac{1}{\alpha})}\cdot\big(  4{\pi}^3{\alpha}^{-3}{y}^{3}\cdot(1- \overline{\epsilon}_{4})
+12{\pi}^2{y}^2+21\pi{\alpha}^{-1}{y}\cdot(1- \overline{\epsilon}_{1})
-4{\pi}^3{\alpha}{y}^3\\
-&4{\pi}^3{\alpha}^{-1}{y}^3-22{\pi}^2{\alpha}^{-2}{y}^2\big)+\frac{8}{3}e^{-\pi\alpha{y}}\big( (3{\pi}^2{\alpha}^2{y}^2+\frac{3}{2}\pi\alpha{y}+\frac{3}{4}) \cdot(1- \overline{\epsilon}_{0})-2{\pi}^3{\alpha}^3{y}^3\big).
\endaligned\end{equation}
Based on the estimates of the four components in $I_{a}(\alpha;y)$ in Lemma \ref{4add1}, one has
\begin{equation}\aligned\label{beuu}
I_{a}(\alpha;y)=\sum_{k=1}^{4}I_{a,k}(\alpha;y)\geq\frac{3}{8}{\alpha}^2 \cdot \mathcal{P}(\alpha;y).
\endaligned\end{equation}
Then,\:the result follows by \eqref{beuu} and Lemma \ref{addadd}.
\end{proof}

\begin{lemma}\label{addadd}For $(\alpha,y)\in \Omega_{1}=\left\{(\alpha,y) |\:\alpha\geq\frac{3}{2},y\geq \frac{4}{5}\alpha\right\}$, it holds that
  \begin{equation}\aligned\nonumber
\mathcal{P}(\alpha;y)\geq \frac{1}{5}.
\endaligned\end{equation}
\end{lemma}
\begin{proof}
For convenience, we denote that
\begin{equation}\aligned\nonumber
g(t):=1+2 e^{-\pi {t}}  \big( 1+\frac{44}{3}{\pi}^2{t}^2 -14\pi t\cdot(1+\overline{\epsilon}_{2})-\frac{8}{3}{\pi}^3{t}^3\cdot(1+ \overline{\epsilon}_{3})\big)+h(\frac{3}{2};t).
\endaligned\end{equation}
Let $r=\alpha $ and $t=\frac{y}{\alpha}$ , then $ r\geq \frac{3}{2},\: t\geq \frac{4}{5}$. Recalling \eqref{wen}, one has
\begin{equation}\aligned\label{hopefuld}
\mathcal{P}(\alpha;y)&=\mathcal{P}(r;rt)=1+2 e^{-\pi {t}}  \big( 1+\frac{44}{3}{\pi}^2{t}^2 -14\pi t\cdot(1+\overline{\epsilon}_{2})-\frac{8}{3}{\pi}^3{t}^3\cdot(1+ \overline{\epsilon}_{3})\big)+h(r;t),
\endaligned\end{equation}
where
\begin{equation}\aligned\label{hopeful}
h(r;t):=&\frac{8}{3}e^{-\pi {t}(r^2+1)}\big(  4{\pi}^3{t}^3\cdot(1- \overline{\epsilon}_{4})+12{\pi}^2{r}^2{t}^2+21\pi{t}(1- \overline{\epsilon}_{1})
-4{\pi}^3{r}^4{t}^3-4{\pi}^3{r}^{2}{t}^3\\
&-22{\pi}^2{t}^2\big)+\frac{8}{3}e^{-{\pi}t{r}^2}\big( (3{\pi}^2{r}^4{t}^2+\frac{3}{2}r^2t+\frac{3}{4}) \cdot(1- \overline{\epsilon}_{0})-2{\pi}^3{r}^6{t}^3\big).
\endaligned\end{equation}
Here, based on Lemma \ref{4add1} (in the following), one has $\overline{\epsilon}_{0}\leq \frac{9}{50},\:\overline{\epsilon}_{1}\leq \frac{1}{470},\:\overline{\epsilon}_{2}\leq \frac{1}{470},\:\overline{\epsilon}_{3}\leq \frac{1}{29},\:\overline{\epsilon}_{4}\leq\frac{1}{29}.$\\
A direct checking (see Figure \ref{hg}) shows that, as $r\geq \frac{3}{2},\: t\geq \frac{4}{5} $,\:$\frac{\partial}{\partial{r}}h(r;t)\geq0$. Then one has
\begin{equation}\aligned\label{cheng}
h(r;t)\geq h(\frac{3}{2};t).
\endaligned\end{equation}
By \eqref{hopefuld}-\eqref{cheng}, one gets
\begin{equation}\aligned\nonumber
\mathcal{P}(\alpha;y)&=\mathcal{P}(r;rt)\geq g(t).
\endaligned\end{equation}
As $t\geq \frac{4}{5}$, the equation $g'(t)=0$ that satisfy the condition $g''(t)>0$ has only one solution $t_{0}=1.781450608\cdots$(see Figure\:\ref{hg}). Thus we can get $ g(t)\geq g(t_{0})\geq0.2141862029\geq \frac{1}{5}$.
\end{proof}

\begin{figure}
\centering
 \includegraphics[scale=0.5]{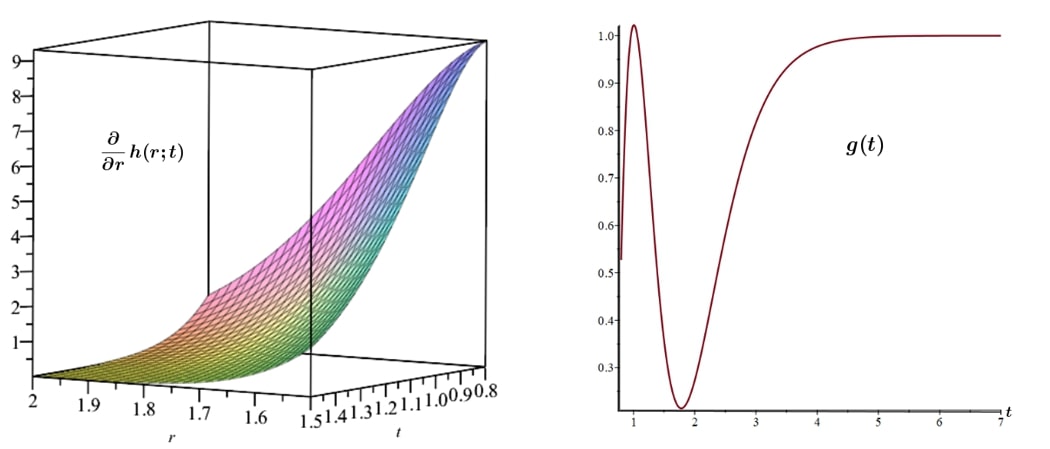}
 \caption{The images of $\frac{\partial}{\partial{r}}h(r;t)$ and\:$g(t)$.}
 \label{hg}
\end{figure}

To better illustrate Lemma \ref{mean4} which states that $I_{b}(\alpha;y)$ is much smaller than $I_{a}(\alpha;y)$, we give the following notations
\begin{equation}\aligned\label{search}
J_{k}(\alpha;y):={\alpha}^{-2}I_{b,k}(\alpha;y),\:\:k=1,2,3,4.
\endaligned\end{equation}

\begin{lemma}\label{mean4} Assume that $(\alpha,y)\in \Omega_{1}$, then
  \begin{equation}\aligned\nonumber
\left| \frac{I_{b}(\alpha;y)}{I_{a}(\alpha;y)} \right|\leq 10^{-4}.
\endaligned\end{equation}
\end{lemma}
\begin{proof}
By Lemma \ref{maea3}, $I_{a}(\alpha;y)\geq\frac{3}{40}{\alpha}^2$, so it suffices to prove $ \left| \frac{I_{b}(\alpha;y)}{\frac{3}{40}{\alpha}^2} \right|\leq10^{-4} $.\\
Based on the relationship between \eqref{mean0} and \eqref{search}, one has
\begin{equation}\aligned\label{eng}
\left|\frac{I_{b}(\alpha;y)}{\frac{3}{40}{\alpha}^2}\right| &=\left|\frac{40}{3}\underset{k=1}{\overset{4}{\sum}}J_{k}(\alpha;y)\right|
\leq \frac{40}{3}\underset{k=1}{\overset{4}{\sum}}\left|J_{k}(\alpha;y)\right|.
\endaligned\end{equation}
Therefore, the result follows by \eqref{eng} and Lemma \ref{4add5}.
\end{proof}

With Lemmas \ref{maea3} and \ref{mean4}, we are ready to prove Lemma \ref{mean5}, which is the central task in this Subsection.
 \begin{proof} \textbf{Proof of Lemma \ref{mean5}.} By Lemmas \ref{maea3} and \ref{mean4}, one has
\begin{equation}\aligned\nonumber
\frac{\partial}{\partial{y}}\mathcal{R}(\alpha;\frac{1}{2}+iy) &=\frac{1}{{\pi}^2}{\alpha}^{-\frac{9}{2}}{y}^{-\frac{1}{2}}(I_{a}(\alpha;y)+I_{b}(\alpha;y))\geq\frac{1}{{\pi}^2}{\alpha}^{-\frac{9}{2}}{y}^{-\frac{1}{2}}I_{a}(\alpha;y)\cdot \Big(1-\left|\frac{I_{b}(\alpha;y)}{I_{a}(\alpha;y)}\right|\Big)\\
&\geq\frac{3}{40{\pi}^2} {\alpha}^{-\frac{5}{2}}{y}^{-\frac{1}{2}}\cdot (1-10^{-4})>0.
\endaligned\end{equation}
\end{proof}

In the rest of this Subsection, to estimate four parts in \eqref{beuu}: ${I}_{a,m}(\alpha;y),\:m=1,2,3,4$ and four functions in \eqref{eng} that are $J_{k}(\alpha;y),k=1,2,3,4$, we give the following two auxiliary Lemmas. To make it clear to state the estimates, we denote that
\begin{equation}\aligned\label{lem4.7eq0}
\overline{\epsilon}_{0}&:=2 \sum_{n=1}^{\infty}e^{-\pi n^2\frac{y}{\alpha}}(-1)^{n+1},\:\:
\overline{\epsilon}_{1}:=\sum_{n=2}^{\infty}{n}^2e^{-\pi({n}^2-1)\frac{y}{\alpha}}(-1)^{n},\:\:
\overline{\epsilon}_{2}:=\sum_{n=2}^{\infty}n^2e^{-\pi (n^2-1)\frac{y}{\alpha}},\\
\overline{\epsilon}_{3}&:=\sum_{n=2}^{\infty}n^6e^{-\pi (n^2-1)\frac{y}{\alpha}},\:\:\:\:\:\:\:\:\:
\overline{\epsilon}_{4}:=\sum_{n=2}^{\infty}n^6e^{-\pi (n^2-1)\frac{y}{\alpha}}(-1)^{n}.
\endaligned\end{equation}

\begin{lemma}[The estimate of $I_{a,m}(\alpha;y),m=1,2,3,4$]\label{4add1}
For $(\alpha,y)\in \Omega_{1}$ $=\left\{(\alpha,y) |\:\alpha\geq\frac{3}{2},y\geq \frac{4}{5}\alpha\right\}$, it holds that
\begin{equation}\aligned\nonumber
(1)\:\: I_{a,1}(\alpha;y)\geq &\frac{3}{8}{\alpha}^2+\frac{3}{4}{\alpha}^2e^{-\pi\frac{y}{\alpha}}+\big((3{\pi}^2{\alpha}^4{y}^2+\frac{3\pi}{2}{\alpha}^3y+\frac{3}{4}{\alpha}^2    )\cdot(1- \overline{\epsilon}_{0})-2{\pi}^3{\alpha}^5{y}^3\big)e^{-\pi\alpha y};\\
(2)\:\:I_{a,2}(\alpha;y)\geq &-\frac{21}{2}\pi\alpha{y}e^{-\pi\frac{y}{\alpha}}\cdot(1+\overline{\epsilon}_{2})+\big((12{\pi}^2{\alpha}^2{y}^2+21\pi\alpha{}y)
\cdot (1-\overline{\epsilon}_{1})-4{\pi}^3{\alpha}^3{y}^3\big)e^{-\pi{y}(\alpha+\frac{1}{\alpha})};\\
(3)\:\:I_{a,3}(\alpha;y)\geq &11{\pi}^2{y}^2e^{-\pi\frac{y}{\alpha}}-(4{\pi}^3\alpha{y}^3+22{\pi}^2{y}^2)e^{-\pi{y}(\alpha+\frac{1}{\alpha})};\\
(4)\:\:I_{a,4}(\alpha;y)\geq & -2{\pi}^3{\alpha}^{-1}{y}^3e^{-\pi\frac{y}{\alpha}}\cdot(1+\overline{\epsilon}_{3})+4{\pi}^3{\alpha}^{-1}{y}^3e^{-\pi{y}(\alpha+\frac{1}{\alpha})}\cdot(1-\overline{\epsilon}_{4}),
\endaligned\end{equation}
where $I_{a,m}(\alpha;y)\:(m=1,2,3,4)$ are defined in \eqref{meani} and $\overline{\epsilon}_{i}\:(i=0,1,2,3,4)$ are defined in \eqref{lem4.7eq0}. Here numerically, $\overline{\epsilon}_{0}\leq \frac{9}{50},\:\:\:\:\overline{\epsilon}_{1}\leq \frac{1}{470},\:\:\:\:\overline{\epsilon}_{2}\leq \frac{1}{470},\:\:\:\:
\overline{\epsilon}_{3}\leq \frac{1}{29},\:\:\:\:\overline{\epsilon}_{4}\leq\frac{1}{29}.$
\end{lemma}
\begin{proof}
We will prove the first three items. While for the fourth item, whose estimate is very similar to the first three items, we will omit its details.
By the definition of the classical one-dimensional theta function given by \eqref{TXY}, one has $\vartheta(X;Y),\vartheta_{X}(X;Y),\vartheta_{XX}(X;Y)$ are even functions with respect to $Y$. Furthermore, by \eqref{TXY} and \eqref{lem4.7eq0}, one has
\begin{equation}\aligned\label{summer1}
\vartheta(\frac{y}{\alpha};0)=&1+2\underset{n=1}{\overset{\infty}{\sum}}e^{-{\pi}{n}^2\frac{y}{\alpha}}\geq 1+2e^{-{\pi}\frac{y}{\alpha}},\:\:\vartheta(\frac{y}{\alpha};\frac{1}{2})=
1-\overline{\epsilon}_{0},\\
{\vartheta}_{X}(\frac{y}{\alpha};0)=&-2{\pi}e^{-\pi\frac{y}{\alpha}}\cdot(1+\overline{\epsilon}_{2}),\:\:{\vartheta}_{X}(\frac{y}{\alpha};\frac{1}{2})=
2{\pi}e^{-\pi\frac{y}{\alpha}}\cdot( 1 -\overline{\epsilon}_{1} ),
\endaligned\end{equation}
and
\begin{equation}\aligned\label{check3}
{\vartheta}_{XX}(\frac{y}{\alpha};0)=&2{\pi}^{2}\sum_{n=1}^{\infty}{n}^4 e^{-\pi{n}^2\frac{y}{\alpha}}\geq 2{\pi}^{2}e^{-\pi\frac{y}{\alpha}},\\
{\vartheta}_{XX}(\frac{y}{\alpha};\frac{1}{2})=&2{\pi}^{2}\sum_{n=1}^{\infty}{n}^4 e^{-\pi{n}^2\frac{y}{\alpha}}(-1)^{n}\geq -2{\pi}^{2}e^{-\pi\frac{y}{\alpha}}.\:\:\:\:\:\:\:\:\:\:\:\:\:\:\:\:\:\:\:\:\:\:\:\:\:\:\:\:\:\:\:\:\:\\
\endaligned\end{equation}
Here, since $y\geq\frac{4}{5}\alpha$, by \eqref{lem4.7eq0}, one has $\overline{\epsilon}_{0}=2\underset{n=1}{\overset{\infty}{\sum}}e^{-{\pi}{n}^2\frac{y}{\alpha}}(-1)^{n+1}\leq 2e^{-{\pi}\frac{y}{\alpha}}\leq 2e^{-\frac{4}{5}{\pi}}<\frac{9}{50}$, and
\begin{equation}\aligned\label{check2}
0<\overline{\epsilon}_{1}<\overline{\epsilon}_{2}&=\sum_{n=2}^{\infty} {n}^2e^{-\pi({n}^2-1)\frac{y}{\alpha}}\leq\sum_{n=2}^{\infty} {n}^2e^{-\frac{4}{5}\pi({n}^2-1)}\leq \frac{1}{470}.
\endaligned\end{equation}
(1) We first estimate $I_{a,1}(\alpha;y)$. By \eqref{meani} and \eqref{summer1},
 \begin{equation}\aligned\nonumber
I_{a,1}(\alpha;y)
=& \frac{3}{8}{\alpha}^{2}\vartheta(\frac{y}{\alpha};0)+\big(  -2{\pi}^3{\alpha}^5{y}^3
+3{\pi}^2{\alpha}^4{y}^2+\frac{3}{2}\pi{\alpha}^3y+\frac{3}{4}{\alpha}^2 \big)e^{-\pi\alpha{y}}\vartheta(\frac{y}{\alpha};\frac{1}{2})\\
\geq& \frac{3}{8}{\alpha}^{2}+\frac{3}{4}{\alpha}^2e^{-\pi\frac{y}{\alpha}}-2{\pi}^3{\alpha}^5{y}^3e^{-\pi\alpha{y}}
+\big(3{\pi}^2{\alpha}^4{y}^2+\frac{3}{2}\pi{\alpha}^3y+\frac{3}{4}{\alpha}^2 \big)\cdot (1-\overline{\epsilon}_{0})e^{-\pi\alpha{y}}.\\
\endaligned\end{equation}
(2) For $I_{a,2}(\alpha;y)$, by \eqref{meani}, \eqref{summer1} and \eqref{check2},
\begin{equation}\aligned\nonumber
I_{a,2}(\alpha;y)
&=\frac{21}{4}{\alpha}{y} \: {\vartheta}_{X}(\frac{y}{\alpha};0)+\big(-2{\pi}^2{\alpha}^{3}{y}^{3}+6\pi{\alpha}^2{y}^2+\frac{21}{2}\alpha{y}    \big)e^{-\pi\alpha{y}}{\vartheta}_{X}(\frac{y}{\alpha};\frac{1}{2})\\
&\geq -\frac{21}{2}\pi\alpha{y}(1+\overline{\epsilon}_{2})e^{-\pi\frac{y}{\alpha}}
+\big((12{\pi}^{2}{\alpha}^{2}{y}^{2}+21\pi\alpha{y})(1-\overline{\epsilon}_{1})-4{\pi}^{3}{\alpha}^{3}{y}^{3}\big)e^{-\pi{y}(\alpha+\frac{1}{\alpha})}.
\endaligned\end{equation}
(3) For $I_{a,3}(\alpha;y)$, by \eqref{meani} and \eqref{check3}, one has
\begin{equation}\aligned\nonumber
I_{a,3}(\alpha;y)
&=\frac{11}{2}{y}^{2}\vartheta_{XX}(\frac{y}{\alpha};0)+(2\pi\alpha{y}^{3}+11{y}^2)e^{-\pi\alpha{y}}\vartheta_{XX}(\frac{y}{\alpha};\frac{1}{2})\\
&\geq 11{\pi}^2{y}^2e^{-\pi\frac{y}{\alpha}}-(4{\pi}^3\alpha{y}^3+22{\pi}^2{y}^2)e^{-\pi{y}(\alpha+\frac{1}{\alpha})}.
\endaligned\end{equation}
\end{proof}

\begin{lemma}[The estimate of $J_{k}(\alpha;y),k=1,2,3,4$]\label{4add5}
 For $(\alpha,y)\in \Omega_{1}=\left\{(\alpha,y) |\alpha\geq\frac{3}{2},y\geq \frac{4}{5}\alpha\right\}$,\:it holds that
 \begin{itemize}
  \item [(1)] $\left|J_{1}(\alpha;y)\right|=\left|\underset{\left|n \right|\geq 2}{\sum}(-{\pi}^3{\alpha}^3{y}^3{n}^6
+\frac{3}{2}{\pi}^2{\alpha}^2{y}^2{n}^4+\frac{3}{4}\pi{\alpha}y{n}^2+\frac{3}{8})e^{-{\pi}{\alpha}y{n}^2 }\vartheta(\frac{y}{\alpha};\frac{n}{2})\right|
\leq  10^{-5};$
      \item [(2)] $\left|J_{2}(\alpha;y)\right|=\left|\underset{\left|n \right|\geq 2}{\sum}(-{\pi}^2{\alpha}{y}^3{n}^4
+3\pi{y}^2{n}^2+\frac{21}{4}{\alpha}^{-1}y)e^{-{\pi}{\alpha}y{n}^2 }\vartheta_{X}(\frac{y}{\alpha};\frac{n}{2})\right|\leq  10^{-7};$
\item [(3)] $\left|J_{3}(\alpha;y)\right|=\left|\underset{\left|n \right|\geq 2}{\sum}(\pi{\alpha}^{-1}{y}^3{n}^2
+\frac{11}{2}{\alpha}^{-2}{y}^2)e^{-{\pi}{\alpha}y{n}^2 }\vartheta_{XX}(\frac{y}{\alpha};\frac{n}{2})\right|\leq 10^{-8};$
\item [(4)] $\left|J_{4}(\alpha;y)\right|=\left|\underset{\left|n \right|\geq 2}{\sum}{\alpha}^{-3}{y}^3e^{-{\pi}{\alpha}y{n}^2 }\vartheta_{XXX}(\frac{y}{\alpha};\frac{n}{2})\right|\leq 10^{-8}.$
  \end{itemize}
  \end{lemma}
\begin{proof}
The estimates of the third and fourth items are similar to those of the first and second items, hence we only give the proof of the first and second items.
By \eqref{TXY}, one has $\vartheta(X;Y),\vartheta_{X}(X;Y)$ are even functions with respect to $Y$, and since $\frac{y}{\alpha}\geq\frac{4}{5}$, one has
\begin{equation}\aligned\label{check6}
\vartheta(\frac{y}{\alpha};\frac{n}{2})&= 1+2 \sum_{m=1}^{\infty} e^{-\pi {m}^2\frac{y}{\alpha}}(-1)^{mn}\geq1-2 \sum_{m=1}^{\infty} e^{-\pi {m}^2\frac{y}{\alpha}}\geq1-2 \sum_{m=1}^{\infty} e^{-\pi {m}^2\frac{4}{5}} >0,\:\:\forall n\in \mathbb{Z}.\\
\vartheta(\frac{y}{\alpha};1)&=1+2 \sum_{m=1}^{\infty} e^{-\pi {m}^2\frac{y}{\alpha}}\leq1+2 \sum_{m=1}^{\infty} e^{-\pi {m}^2\frac{4}{5}}\leq \frac{6}{5},  \:\:\:\: \:\:\:\:\frac{\vartheta(\frac{y}{\alpha};\frac{n}{2})}{\vartheta(\frac{y}{\alpha};1)}\leq1.
\endaligned\end{equation}
Similar to the analysis of \eqref{check6}, by \eqref{TXY}, as $t\geq\frac{4}{5}$, one gets
\begin{equation}\aligned\label{atime1}
-{\vartheta}_{X}(t;1)&=2\pi \sum_{m=1}^{\infty}m^2 e^{-\pi {m}^2{t}}\leq2\pi \sum_{m=1}^{\infty}m^2 e^{-\frac{4}{5}\pi {m}^2}\leq \frac{3}{5},\\
{\vartheta}_{X}(t;\frac{n}{2})&= -2\pi \sum_{m=1}^{\infty}m^2 e^{-\pi {m}^2{t}}(-1)^{mn},\:\:\:\:\:\frac{\vartheta_{X}(t;\frac{n}{2})}{\vartheta_{X}(t;1)}\leq 1.
\endaligned\end{equation}
For simplicity, we denote that
\begin{equation}\aligned\nonumber
E_{1}(\alpha;y):=&128{\pi}^3{\alpha}^3{y}^3e^{-4{\pi}{\alpha}y }\cdot \vartheta(\frac{y}{\alpha};1) \cdot\Big(1+\sum_{n=3}^{\infty}\frac{{n}^6}{2^6}e^{-{\pi}{\alpha}y({n}^2 -4) }\Big).\\
E_{2}(r;t):=&32{\pi}^2{r}^4{t}^3 e^{-4 {\pi}{r}^2t  }\cdot(-{\vartheta}_{X}(t;1) \cdot \Big(1+\sum_{n=3}^{\infty}\frac{{n}^4}{2^4}e^{-{\pi}{r}^2t ({n}^2-4)}\Big).
\endaligned\end{equation}
(1) For $J_{1}(\alpha;y)$, since $\alpha\geq\frac{3}{2},y\geq\frac{4}{5}\alpha\geq\frac{6}{5}$, and by \eqref{check6}, one has
\begin{equation}\aligned\label{check7}
\left|J_{1}(\alpha;y)\right|&=\left|\underset{\left|n \right|\geq 2}{\sum}(-{\pi}^3{\alpha}^3{y}^3{n}^6
+\frac{3}{2}{\pi}^2{\alpha}^2{y}^2{n}^4+\frac{3}{4}\pi{\alpha}y{n}^2+\frac{3}{8})e^{-{\pi}{\alpha}y{n}^2 }\vartheta(\frac{y}{\alpha};\frac{n}{2})\right|\\
&\leq 2\underset{n\geq 2}{\sum}{\pi}^3{\alpha}^3{y}^3{n}^6e^{-{\pi}{\alpha}y{n}^2 }\vartheta(\frac{y}{\alpha};\frac{n}{2})\\
&\leq128{\pi}^3{\alpha}^3{y}^3e^{-4{\pi}{\alpha}y }\vartheta(\frac{y}{\alpha};1)\cdot\Big(1+\sum_{n=3}^{\infty}\frac{{n}^6}{2^6}e^{-{\pi}{\alpha}y({n}^2 -4) }\frac{\vartheta(\frac{y}{\alpha};\frac{n}{2})}{\vartheta(\frac{y}{\alpha};1)}\Big).
\endaligned\end{equation}
Combining \eqref{check6}, \eqref{check7} with $\alpha\geq\frac{3}{2},y\geq\frac{6}{5}$, one has
$$\left|J_{1}(\alpha;y)\right|\leq E_{1}(\alpha;y)\leq E_{1}(\frac{3}{2};\frac{6}{5})\leq 10^{-5}.$$
(2) For $J_{2}(\alpha;y)$, let $r=\alpha,\:t=\frac{y}{\alpha}$, then $r \geq \frac{3}{2},\:t\geq \frac{4}{5}$. By \eqref{TXY} and \eqref{atime1}, one has
 \begin{equation}\aligned\label{atime2}
 \left|J_{2}(\alpha;y)\right|&=\left|J_{2}(r;rt)\right|=2 \left|\underset{n\geq 2}{\sum}(-{\pi}^2{r}^4{t}^3{n}^4
+3\pi{r}^2{t}^2{n}^2+\frac{21}{4}t)e^{-{\pi}{r}^2t {n}^2 }\vartheta_{X}(t;\frac{n}{2})\right|\\
&\leq 2 \left|\underset{n\geq 2}{\sum}-{\pi}^2{r}^4{t}^3{n}^4e^{-{\pi}{r}^2t {n}^2 }\vartheta_{X}(t;\frac{n}{2})\right|\\
&\leq 32{\pi}^2{r}^4{t}^3 e^{-4 {\pi}{r}^2t  }\big(-\vartheta_{X}(t;1)\big)\cdot \Big(1+\sum_{n=3}^{\infty}\frac{{n}^4}{2^4}e^{-{\pi}{r}^2t ({n}^2-4) }\frac{\vartheta_{X}(t;\frac{n}{2})}{\vartheta_{X}(t;1)}\Big).\\
\endaligned\end{equation}
Combining \eqref{atime1}, \eqref{atime2} with $r \geq \frac{3}{2},\:t\geq \frac{4}{5}$, one gets
$$\left|J_{2}(\alpha;y)\right|\leq E_{2}(r;t)\leq E_{2}(\frac{3}{2};\frac{4}{5})\leq 10^{-7}.$$
\end{proof}

\subsection{Region $\Omega_{2}$: estimate of $\Big(\frac{{\partial}^2}{\partial{y^2}}+\frac{2}{y}\frac{\partial}{\partial{y}}\Big)\mathcal{R}(\alpha;\frac{1}{2}+iy)$}

In this Subsection,\:we aim to prove that
\begin{lemma}\label{4lema13} Assume that $(\alpha;y)\in \Omega_{2}=\left\{(\alpha,y) |\:\:\alpha\geq\frac{3}{2},y\in[\frac{\sqrt{3}}{2},\frac{4}{5}\alpha]\right\}$,\:it holds that
\begin{equation}\aligned\nonumber
\Big(\frac{{\partial}^2}{\partial{y^2}}+\frac{2}{y}\frac{\partial}{\partial{y}}\Big)\mathcal{R}(\alpha;\frac{1}{2}+iy)\geq \frac{77}{50{y}^4}e^{-\pi\frac{\alpha}{y}}>0.
\endaligned\end{equation}
\end{lemma}

To prove Lemma \ref{4lema13}, based on the expression of $\mathcal{R}(\alpha;z)$ given by \eqref{RR}, by a direct computation, we first provide the expression for $\Big(\frac{{\partial}^2}{\partial{y^2}}+\frac{2}{y}\frac{\partial}{\partial{y}}\Big)\mathcal{R}(\alpha;\frac{1}{2}+iy)$ as follows.

\begin{lemma}[An identity for $\big(\frac{{\partial}^2}{\partial{y^2}}+\frac{2}{y}\frac{\partial}{\partial{y}}\big)\mathcal{R}(\alpha;\frac{1}{2}+iy)$]
\label{exprr}
For $\alpha,\:y>0$,\:it holds that
  \begin{equation}\aligned\nonumber
\Big(\frac{{\partial}^2}{\partial{y^2}}+\frac{2}{y}\frac{\partial}{\partial{y}}\Big)\mathcal{R}(\alpha;\frac{1}{2}+iy)
=&{\pi}^2{\alpha}^2  \mathcal{W}_{a}(\alpha;y)
+2 \mathcal{W}_{b}(\alpha;y)+\frac{4}{y}\mathcal{W}_{c}(\alpha;y)\\
&-4\pi\alpha \mathcal{W}_{d}(\alpha;y)-\frac{2\pi\alpha}{y}\mathcal{W}_{e}(\alpha;y),\\
\endaligned\end{equation}
where
\begin{equation}\aligned\label{4lem3}
\mathcal{W}_{a}(\alpha;y)&:=\underset{n,m}{\sum}(n^2-\frac{(m+\frac{n}{2})^2}{y^2})^2(yn^2+\frac{(m+\frac{n}{2})^2}{y})^2
e^{-\pi\alpha(yn^2+\frac{(m+\frac{n}{2})^2}{y})},\\
\mathcal{W}_{b}(\alpha;y)&:= \underset{n,m}{\sum}(n^2-\frac{(m+\frac{n}{2})^2}{y^2})^2e^{-\pi\alpha(yn^2+\frac{(m+\frac{n}{2})^2}{y})},\\
\mathcal{W}_{c}(\alpha;y)&:=\underset{n,m}{\sum}{n}^2(yn^2+\frac{(m+\frac{n}{2})^2}{y})e^{-\pi\alpha(yn^2+\frac{(m+\frac{n}{2})^2}{y})},\\
\mathcal{W}_{d}(\alpha;y)&:=\underset{n,m}{\sum}(n^2-\frac{(m+\frac{n}{2})^2}{y^2})^2(yn^2+\frac{(m+\frac{n}{2})^2}{y})
e^{-\pi\alpha(yn^2+\frac{(m+\frac{n}{2})^2}{y})},\\
\mathcal{W}_{e}(\alpha;y)&:=\underset{n,m}{\sum}{n}^2(yn^2+\frac{(m+\frac{n}{2})^2}{y})^2e^{-\pi\alpha(yn^2+\frac{(m+\frac{n}{2})^2}{y})}.\\
\endaligned\end{equation}
\end{lemma}

There are five types of double sums appeared in Lemma \ref{exprr}. By the lower bound estimates of $\mathcal{W}_{a},\:J_{b},\:\mathcal{W}_{c}$ and the upper bound estimates of $\mathcal{W}_{d},\:\mathcal{W}_{e}$ that will be placed at the end of this Subsection, we can derive a lower bound for $\big(\frac{\partial^2}{\partial y^2}+\frac{2}{y}\frac{\partial}{\partial y}\big)\mathcal{R}(\alpha;\frac{1}{2}+iy)$ in Lemma \ref{4lemm10} which will facilitate the proof of Lemma \ref{4lema13}.

\begin{lemma}[The lower bound of $\big(\frac{\partial^2}{\partial y^2}+\frac{2}{y}\frac{\partial}{\partial y}\big)\mathcal{R}(\alpha;\frac{1}{2}+iy)$]

  \label{4lemm10}
  Assume that $(\alpha;y)\in \Omega_{2}$,\:it holds that
  \begin{equation}\aligned\nonumber
\big(\frac{\partial^2}{\partial y^2}+\frac{2}{y}\frac{\partial}{\partial y}\big)\mathcal{R}(\alpha;\frac{1}{2}+iy)
\geq\frac{4}{y^4}e^{-\pi\frac{\alpha}{y}}\cdot
\mathcal{Y}(\alpha;y),
\endaligned\end{equation}
where
\begin{equation}\aligned\label{even}
\mathcal{Y}(\alpha;y):=1&+\frac{{\pi}^2{\alpha}^2}{2{y}^2}-\frac{2\pi\alpha}{y}(1+{\epsilon}_{a})+y^4e^{-\pi\alpha(y-\frac{3}{4y})}\cdot\Big( {\pi}^2{\alpha}^2{y}^2 (1-\frac{1}{16{y}^4})^2 +2(1-\frac{1}{4{y}^2})^2 \\
&+4+\frac{\pi\alpha}{y}+\frac{1}{{y}^2}+\frac{\pi\alpha}{4{y}^3}-4\pi\alpha{y}\cdot(1+{\epsilon}_{b})-2\pi\alpha y\cdot(1+\epsilon_{c})\Big).
\endaligned\end{equation}
Here ${\epsilon}_a$ is small and consists of four smaller parts
$$
{\epsilon}_a:=\epsilon_{a,1}+\epsilon_{a,2}+\epsilon_{a,3}+\epsilon_{a,4},
$$
and each $\epsilon_{a,j}(j=1,2,3,4)$ can be expressed by
\begin{equation}\aligned\nonumber
\epsilon_{a,1}:&=2\underset{n=1}{\overset{\infty}{\sum}}e^{-4{\pi}{\alpha}y{n}^2},\:\:\:\:\:
\epsilon_{a,2}:=\underset{n=2}{\overset{\infty}{\sum}}n^6e^{-\pi\alpha\frac{n^2-1}{y}},\:\:\:\:
\epsilon_{a,3}:=\epsilon_{a,1}\cdot\epsilon_{a,2},\\
\epsilon_{a,4}:&=64\:{y}^6e^{-{\pi}{\alpha}(4y-\frac{1}{y})}(1+\underset{n=2}{\overset{\infty}{\sum}}e^{-4{\pi}{\alpha}y(n^2-1)})\cdot\vartheta_3(\frac{\alpha}{y}).\\
\endaligned\end{equation}
Here ${\epsilon}_b$ is small and consists of four smaller parts
\begin{equation}\aligned\nonumber
{\epsilon}_b:=\epsilon_{b,1}+\epsilon_{b,2}+\epsilon_{b,3}+\epsilon_{b,4},
\endaligned\end{equation}
and
\begin{equation}\aligned\nonumber
{\epsilon}_{b,1}:&=\underset{n=2}{\overset{\infty}{\sum}}(2n-1)^6e^{-4{\pi}{\alpha}y(n-1)\cdot n},\:\:\:\:
{\epsilon}_{b,2}:=\underset{n=2}{\overset{\infty}{\sum}}e^{-{\pi}{\alpha}\frac{(n-1)\cdot n}{y}},\:\:\:
{\epsilon}_{b,3}:={\epsilon}_{b,1}\cdot {\epsilon}_{b,2},\\
{\epsilon}_{b,4}:&=\frac{1}{64{y}^6}(1+\underset{n=2}{\overset{\infty}{\sum}}e^{-4{\pi}{\alpha}y(n-1)\cdot n})(1+\underset{n=2}{\overset{\infty}{\sum}}(2n-1)^{6}e^{-{\pi}{\alpha}\frac{(n-1)\cdot n}{y}}).\\
\endaligned\end{equation}
Here ${\epsilon}_c$ is small and consists of eight smaller parts
$$
{\epsilon}_{c}:=\underset{j=1}{\overset{3}{\sum}}\epsilon_{b,j}+\underset{j=1}{\overset{5}{\sum}}\epsilon_{c,j},
$$
and each $\epsilon_{c,j}(j=1,2,3,4,5)$ can be expressed by
\begin{equation}\aligned\nonumber
&\epsilon_{c,1}:=32e^{-\pi\alpha(3y-\frac{1}{4y})}\cdot(1+\underset{n=2 }{\overset{\infty}{\sum}}n^{6}e^{-4\pi\alpha{y}({n}^2-1)})\cdot\vartheta_3(\frac{\alpha}{y}),\\
&\epsilon_{c,2}:=\frac{4}{y^4}e^{-3\pi\alpha(y+\frac{1}{4y})}(1+\underset{n=2}{\overset{\infty}{\sum}}n^{2}e^{-4{\pi}{\alpha}y(n^2-1)})
(1+\underset{m=2}{\overset{\infty}{\sum}}m^{4}e^{-\pi\alpha\frac{m^2-1}{y}}),\\
&\epsilon_{c,3}:=\frac{1}{16y^4}\big(1+\underset{n=2}{\overset{\infty}{\sum}}(2n-1)^2e^{-4{\pi}{\alpha}y(n-1)\cdot n} \big)\cdot\big(1+\underset{m=2}{\overset{\infty}{\sum}}(2m-1)^4e^{-\pi\alpha\frac{(m-1)\cdot m}{y}} \big),\\
&\epsilon_{c,4}:=\frac{32}{y^2}e^{-3\pi\alpha(y+\frac{1}{4y})}(1+\underset{n=2}{\overset{\infty}{\sum}}n^{4}e^{-4{\pi}{\alpha}y(n^2-1)})
(1+\underset{m=2}{\overset{\infty}{\sum}}m^{2}e^{-\pi\alpha\frac{m^2-1}{y}}),\\
&\epsilon_{c,5}:=\frac{1}{2y^2}\big(1+\underset{n=2}{\overset{\infty}{\sum}}(2n-1)^4e^{-4{\pi}{\alpha}y(n-1)\cdot n} \big)\cdot\big(1+\underset{m=2}{\overset{\infty}{\sum}}(2m-1)^2e^{-\pi\alpha\frac{(m-1)\cdot m}{y}} \big).\\
\endaligned\end{equation}
\end{lemma}
The subsequent Lemma \ref{4lema11} will provide the numerical upper bound estimates, further elucidating that ${\epsilon}_{a},\:{\epsilon}_{b},\:{\epsilon}_{c}$ are small in Lemma \ref{4lemm10}, thereby laying down the foundation for proving Lemma \ref{4lema12}.
\begin{lemma}\label{4lema11}Assume that $(\alpha,y)\in \Omega_{2}=\left\{(\alpha,y) |\:\:\alpha\geq\frac{3}{2},y\in[\frac{\sqrt{3}}{2},\frac{4}{5}\alpha]\right\}$, it holds that
\begin{equation}\aligned\nonumber
{\epsilon}_{a}\leq \frac{1}{980},\:\:\:\:
{\epsilon}_{b}\leq \frac{101}{2100},\:\:\:\:
\epsilon_{c}\leq \frac{4}{5}.
\endaligned\end{equation}
\end{lemma}
\begin{proof}
One needs to use the fact that $\alpha\geq \frac{3}{2},\:y\geq \frac{\sqrt{3}}{2}\:and \:\frac{\alpha}{y}\geq \frac{5}{4}$. The terms of$\:{\epsilon}_{a},\:{\epsilon}_{b},\:{\epsilon}_{c}$ are exponentially decaying and the summation can be controlled effectively thereby.
\end{proof}

\begin{lemma}[The positiveness of lower bound function in Lemma \ref{4lemm10}]\label{4lema12}
Assume that $(\alpha,y)\in \Omega_{2}=\left\{(\alpha,y) |\alpha\geq\frac{3}{2},y\in[\frac{\sqrt{3}}{2},\frac{4}{5}\alpha]\right\}$, then
\begin{equation}\aligned\nonumber
\mathcal{Y}(\alpha;y)\geq \frac{77}{200}>0.
\endaligned\end{equation}
\end{lemma}
\begin{proof}
Notice that $\mathcal{Y}(\alpha;y)$ given by \eqref{even} is decreasing with respect to $y$ as $(\alpha;y)\in \Omega_{2}$ (see Figure \ref{y}). Then one gets
\begin{equation}\aligned\nonumber
\mathcal{Y}(\alpha;y)\geq \mathcal{Y}(\alpha;\frac{4}{5}\alpha).
\endaligned\end{equation}
As $\alpha \geq \frac{3}{2}$, $\mathcal{Y}(\alpha;\frac{4}{5}\alpha)$ starts in an increasing state and then decreases rapidly to a constant (see Figure \ref{y}), one has
\begin{equation}\aligned\nonumber
\mathcal{Y}(\alpha;y)&\geq \mathcal{Y}(\alpha;\frac{4}{5}\alpha)\geq {\min}\{\mathcal{Y}(\frac{3}{2};\frac{6}{5}),\underset{\alpha\rightarrow \infty} {\lim}\mathcal{Y}(\alpha;\frac{4}{5}\alpha)\}\geq\mathcal{Y}(\frac{3}{2};\frac{6}{5})\geq\frac{77}{200}>0.
\endaligned\end{equation}
\end{proof}

\begin{figure}
\centering
 \includegraphics[scale=0.45]{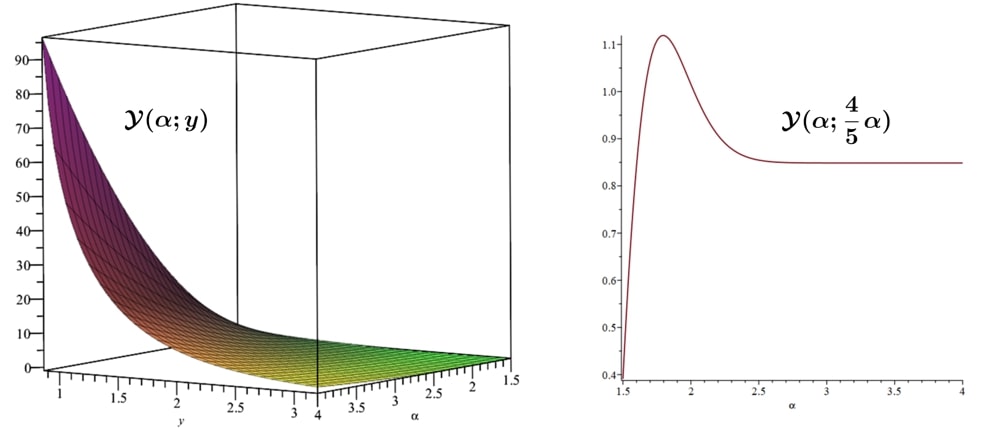}
 \caption{The images of $\mathcal{Y}(\alpha;y)$ and\:$\mathcal{Y}(\alpha;\frac{4}{5}\alpha)$.}
 \label{y}
\end{figure}

Lemmas \ref{4lemm10} and \ref{4lema12} complete the proof of Lemma \ref{4lema13}, which is the main task in this Subsection. In the rest of this Subsection, we shall estimate the four double sums in \eqref{4lem3} respectively to accomplish the proof of Lemma \ref{4lemm10}.

\begin{lemma}[A lower bound estimate of $\mathcal{W}_{a}(\alpha;y)$]\label{4lema4}
 For $ \alpha ,\:y>0$, it holds that
\begin{equation}\aligned\nonumber
\mathcal{W}_{a}(\alpha;y)&=\underset{n,m}{\sum}(n^2-\frac{(m+\frac{n}{2})^2}{y^2})^2(yn^2+\frac{(m+\frac{n}{2})^2}{y})^2
e^{-\pi\alpha(yn^2+\frac{(m+\frac{n}{2})^2}{y})}\\
&\geq \frac{2}{{y}^6}e^{-\frac{\pi\alpha}{y}}+4{y}^2(1-\frac{1}{16y^4})^2e^{-\pi\alpha(y+\frac{1}{4y})}.
\endaligned\end{equation}
\end{lemma}
\begin{proof}
Notice that all terms of the double sum $\mathcal{W}_{a}(\alpha;y)$ are non-negative. Then the term
\begin{equation}\aligned\nonumber
(n,m)\in \{(0,-1),(0,1)  \}\:\:\mathrm{contributes\:an\:amount}\:\frac{1}{{y}^6}e^{-\frac{\pi\alpha}{y}}\:\mathrm{respectively},
\endaligned\end{equation}
and the  term $(n,m)\in \{(1,0),(1,-1),(-1,0),(-1,1) \}$  contributes the following  amount respectively
\begin{equation}\aligned\nonumber
 {y}^2(1-\frac{1}{16y^4})^2e^{-\pi\alpha(y+\frac{1}{4y})}.
\endaligned\end{equation}
\end{proof}

\begin{lemma}[A lower bound estimate of $\mathcal{W}_{b}(\alpha;y)$]\label{4lema5} For $ \alpha ,\:y>0$, it holds that
\begin{equation}\aligned\nonumber
\mathcal{W}_{b}(\alpha;y)&=\underset{n,m}{\sum}(n^2-\frac{(m+\frac{n}{2})^2}{y^2})^2e^{-\pi\alpha(yn^2+\frac{(m+\frac{n}{2})^2}{y})}
\geq \frac{2}{{y}^4}e^{-\frac{\pi\alpha}{y}}+4(1-\frac{1}{4y^2})^2e^{-\pi\alpha(y+\frac{1}{4y})}.
\endaligned\end{equation}
\end{lemma}
\begin{proof}
Notice that all terms of the double sum $\mathcal{W}_{b}(\alpha;y)$ are non-negative. Then the term
\begin{equation}\aligned\nonumber
(n,m)\in \{(0,-1),(0,1)  \}\:\:\mathrm{contributes\:\: an \:\:amount}\:\:\frac{1}{{y}^4}e^{-\frac{\pi\alpha}{y}}\:\:\mathrm{respectively},
\endaligned\end{equation}
and the term $(n,m)\in \{(1,0),(1,-1),(-1,0),(-1,1) \}$  contributes the following  amount respectively
\begin{equation}\aligned\nonumber
(1-\frac{1}{4y^2})^2e^{-\pi\alpha(y+\frac{1}{4y})}.
\endaligned\end{equation}
\end{proof}

\begin{lemma}[A lower bound estimate of $\mathcal{W}_{c}(\alpha;y)$]\label{4lema6} For $ \alpha,\:y>0$, it holds that
\begin{equation}\aligned\nonumber
\mathcal{W}_{c}(\alpha;y)&=\underset{n,m}{\sum}{n}^2(yn^2+\frac{(m+\frac{n}{2})^2}{y})e^{-\pi\alpha(yn^2+\frac{(m+\frac{n}{2})^2}{y})}
\geq (4y+\frac{1}{y})e^{-\pi\alpha(y+\frac{1}{4y})}.
\endaligned\end{equation}
\end{lemma}
\begin{proof}
Notice that all terms of the double sum $\mathcal{W}_{c}(\alpha;y)$ are non-negative. Then the term
\begin{equation}\aligned\nonumber
(n,m)\in \{(1,0),(1,-1),(-1,0),(-1,1) \}\:\:\mathrm{contributes\:\: an \:\:amount}\:(y+\frac{1}{4y})e^{-\pi\alpha(y+\frac{1}{4y})}\:\:\mathrm{respectively}.
\endaligned\end{equation}
\end{proof}

To better estimate an upper bound of $\mathcal{W}_{d}(\alpha;y)$, we give the following notations.
\begin{equation}\aligned\label{staro}
\mathcal{W}_{d,1}:=&\underset{p\equiv q\equiv0(mod2) }{\sum} y {p}^6 e^{-\pi\alpha(y{p}^2+\frac{q^2}{4y})},\:\:\:\:\:
\mathcal{W}_{d,2}:=\underset{p\equiv q\equiv1(mod2) }{\sum}  y{p}^6 e^{-\pi\alpha(y{p}^2+\frac{q^2}{4y})},\\
\mathcal{W}_{d,3}:=&\underset{p\equiv q\equiv0(mod2) }{\sum}\frac{q^6}{64{y}^5} e^{-\pi\alpha(y{p}^2+\frac{q^2}{4y})},\:\:
\mathcal{W}_{d,4}:=\underset{p\equiv q\equiv1(mod2) }{\sum}\frac{q^6}{64{y}^5} e^{-\pi\alpha(y{p}^2+\frac{q^2}{4y})},\\
\mathcal{W}_{d,5}:=&\underset{p\equiv q\equiv1(mod2) }{\sum}\frac{{p}^4{q}^2}{4{y}}e^{-\pi\alpha(y{p}^2+\frac{q^2}{4y})},\:\:
\mathcal{W}_{d,6}:=\underset{p\equiv q\equiv1(mod2) }{\sum} \frac{{p}^2{q}^4}{16{y}^3} e^{-\pi\alpha(y{p}^2+\frac{q^2}{4y})}.\\
\endaligned\end{equation}

\begin{lemma}[An upper bound estimate of $\mathcal{W}_{d}(\alpha;y)$] \label{4lema8} For $ \alpha,\:y>0$, it holds that
  \begin{equation}\aligned\nonumber
\mathcal{W}_{d}(\alpha;y)&= \underset{n,m}{\sum}(n^2-\frac{(m+\frac{n}{2})^2}{y^2})^2(yn^2+\frac{(m+\frac{n}{2})^2}{y})
e^{-\pi\alpha(yn^2+\frac{(m+\frac{n}{2})^2}{y})}\\
&\leq\frac{2}{{y}^5}e^{-\frac{\pi\alpha}{y}}\cdot(1+{\epsilon}_{a})+4ye^{-\pi\alpha(y+\frac{1}{4y})}\cdot(1-\frac{1}{4y^2}-\frac{1}{16y^4}+
{\epsilon}_{b}),
\endaligned\end{equation}
where
\begin{equation}\aligned\label{4lab1}\nonumber
{\epsilon}_{a}:=\sum_{j=1}^{4}\epsilon_{a,j}, \:\:\:\:{\epsilon}_{b}:=\sum_{j=1}^{4}\epsilon_{b,j} .\\
\endaligned\end{equation}
Here each $\epsilon_{a,j}(j=1,2,3,4)$ is small and expressed by
\begin{equation}\aligned\label{eqa}
\epsilon_{a,1}:&=2\underset{n=1}{\overset{\infty}{\sum}}e^{-4{\pi}{\alpha}y{n}^2},\:\:
\epsilon_{a,2}:=\underset{n=2}{\overset{\infty}{\sum}}n^6e^{-\pi\alpha\frac{n^2-1}{y}},\:\:
\epsilon_{a,3}:=\epsilon_{a,1}\cdot\epsilon_{a,2},\\
\epsilon_{a,4}:&=64{y}^6e^{-{\pi}{\alpha}(4y-\frac{1}{y})}(1+\underset{n=2}{\overset{\infty}{\sum}}n^{6}e^{-4{\pi}{\alpha}y(n^2-1)})\cdot\vartheta_3(\frac{\alpha}{y}),\\
\endaligned\end{equation}
where $\vartheta_{3}(X)$ is defined in \eqref{JacobiJJJ} and each $\epsilon_{b,j}(j=1,2,3,4)$ is small and expressed by
\begin{equation}\aligned\label{eqb}
{\epsilon}_{b,1}:&=\underset{n=2}{\overset{\infty}{\sum}}(2n-1)^{6}e^{-4{\pi}{\alpha}y(n-1)\cdot n},\:\:\:\:
{\epsilon}_{b,2}:=\underset{n=2}{\overset{\infty}{\sum}}e^{-{\pi}{\alpha}\frac{(n-1)\cdot n}{y}},\:\:\:\:
{\epsilon}_{b,3}:={\epsilon}_{b,1}\cdot {\epsilon}_{b,2},\\
{\epsilon}_{b,4}:&=\frac{1}{64{y}^6}(1+\underset{n=2}{\overset{\infty}{\sum}}e^{-4{\pi}{\alpha}y(n-1)\cdot n})(1+\underset{n=2}{\overset{\infty}{\sum}}(2n-1)^{6}e^{-{\pi}{\alpha}\frac{(n-1)\cdot n}{y}}).\\
\endaligned\end{equation}
\end{lemma}
\begin{proof}
Firstly, by \eqref{4lem3}, $\mathcal{W}_{d} $ can be rewritten as
\begin{equation}\aligned\label{star1}
\mathcal{W}_{d}(\alpha;y)
&=\underset{p\equiv q(mod2) }{\sum} y(p^4-\frac{q^4}{16{y}^4})(p^2-\frac{q^2}{4y^2})e^{-\pi\alpha(y{p}^2+\frac{q^2}{4y})}\\
&\leq \mathcal{W}_{d,1}+\mathcal{W}_{d,2}+\mathcal{W}_{d,3}+\mathcal{W}_{d,4}-\mathcal{W}_{d,5}-\mathcal{W}_{d,6}.
\endaligned\end{equation}
Next, we shall calculate $\mathcal{W}_{d,1},\:\mathcal{W}_{d,2},\:\mathcal{W}_{d,3},\:\mathcal{W}_{d,4}$ sequentially. By \eqref{staro}, one has
\begin{equation}\aligned\label{star2}
\mathcal{W}_{d,1}&=\underset{p=2n,\:q=2m }{\sum} y{p}^{6}e^{-\pi\alpha(y{p}^2+\frac{q^2}{4y})}
=128 y\underset{n=1 }{\overset{\infty}{\sum}} n^{6}e^{-4\pi\alpha{y}{n}^2}\cdot\underset{m }{\sum}e^{-\pi\alpha\frac{m^2}{y}}\:\:\:\:\:\:\:\:\:\:\:\:\:\:\:\:\:\:\:\:\:\:\:\:\:\:\:\:\:\:\:\:\:\:\:\\
&=128 y e^{-4\pi\alpha{y}}\cdot(1+\underset{n=2 }{\overset{\infty}{\sum}}n^{6}e^{-4\pi\alpha{y}({n}^2-1)})\cdot\vartheta_3(\frac{\alpha}{y})=\frac{2}{{y}^5}e^{-\frac{\pi\alpha}{y}}\cdot\epsilon_{a,4}.
\endaligned\end{equation}
\begin{equation}\aligned\label{star3}
\mathcal{W}_{d,2}&=\underset{p=2n-1,\:q=2m-1 }{\sum} y{p}^{6}e^{-\pi\alpha(y{p}^2+\frac{q^2}{4y})}=4 y\:\underset{n=1 }{\overset{\infty}{\sum}} (2n-1)^{6}e^{-\pi\alpha{y}(2n-1)^2}\cdot \underset{m=1 }{\overset{\infty}{\sum}} e^{-\pi\alpha\frac{(2m-1)^2}{4y}}\\
&=4 y e^{-\pi\alpha(y+\frac{1}{4y})}\cdot(1+\underset{n=2 }{\overset{\infty}{\sum}} (2n-1)^{6}e^{-4\pi\alpha{y}(n-1)n})\cdot(1+\underset{m=2 }{\overset{\infty}{\sum}} e^{-\pi\alpha\frac{(m-1)m}{y}})\\
&=4 y e^{-\pi\alpha(y+\frac{1}{4y})}\cdot(1+{\epsilon}_{b,1}+{\epsilon}_{b,2}+{\epsilon}_{b,3}).\\
\endaligned\end{equation}
\begin{equation}\aligned\label{star4}
\mathcal{W}_{d,3}&=\underset{p=2n,\:q=2m }{\sum} \frac{q^6}{64{y}^5}e^{-\pi\alpha(y{p}^2+\frac{q^2}{4y})}=\frac{1}{y^5}\underset{n }{\sum} e^{-4\pi\alpha{y}{n}^2}\cdot \underset{m }{\sum}m^6 e^{-\pi\alpha\frac{m^2}{y}}\\
&=\frac{2}{y^5}e^{-\pi\frac{\alpha}{y}}\cdot(1+2\underset{n=1}{\overset{\infty}{\sum}}e^{-4{\pi}{\alpha}y{n}^2})
\cdot(1+\underset{m=2}{\overset{\infty}{\sum}}m^6e^{-\pi\alpha\frac{m^2-1}{y}})\:\:\:\:\:\:\:\:\:\:\:\:\:\:\:\:\:\:\:\:\:\:\:\:\:\:\:\:\:\:\:\:\:\:\:\:\:\:\:\:\:\:\:\:\:\:\:\:\:\:\:\:\\
&=\frac{2}{y^5}e^{-\pi\frac{\alpha}{y}}\cdot(1+\epsilon_{a,1})\cdot(1+\epsilon_{a,2})=\frac{2}{y^5}e^{-\pi\frac{\alpha}{y}}\cdot(1+\epsilon_{a,1}+\epsilon_{a,2}+\epsilon_{a,3}).
\endaligned\end{equation}
Similar to the calculation in \eqref{star3}, by \eqref{staro}, one gets $\mathcal{W}_{d,4}$ as follows,
\begin{equation}\aligned\label{star5}
\mathcal{W}_{d,4}=4ye^{-\pi\alpha(y+\frac{1}{4y})}\cdot {\epsilon}_{b,4}.
\endaligned\end{equation}
Thus, the desired result follows by \eqref{star1}-\eqref{star5} and Lemmas \ref{4lemaadd1}-\ref{4lemaadd2}.
\end{proof}

\begin{lemma}[A lower bound estimate of $\mathcal{W}_{d,5}$ in \eqref{staro}]\label{4lemaadd1} For $ \alpha,\:y>0$, it holds that
\begin{equation}\aligned\nonumber
\mathcal{W}_{d,5}\geq \frac{1}{y}e^{-\pi\alpha(y+\frac{1}{4y})}.
\endaligned\end{equation}
\end{lemma}
\begin{proof}
By \eqref{staro}, one has
\begin{equation}\aligned\label{star6}\nonumber
\mathcal{W}_{d,5}=\underset{p=2n-1,q=2m-1 }{\sum} \frac{{p}^4{q}^2}{4{y}}e^{-\pi\alpha(y{p}^2+\frac{q^2}{4y})}=\frac{1}{4{y}}\underset{n,m}{\sum}(2n-1)^4(2m-1)^2e^{-\pi\alpha[y(2n-1)^2+\frac{(2m-1)^2}{4y}]}.
\endaligned\end{equation}
Notice that all terms of the double sum $\mathcal{W}_{d,5}$ are non-negative. Then the term
\begin{equation}\aligned\nonumber
(n,m)\in \{(0,0),(0,1),(1,0),(1,1) \}\:\:\mathrm{contributes\:\: an \:\:amount}\:\frac{1}{4y}e^{-\pi\alpha(y+\frac{1}{4y})}\:\:\mathrm{respectively}.
\endaligned\end{equation}
\end{proof}

\begin{lemma}[A lower bound estimate of $\mathcal{W}_{d,6}$ in \eqref{staro}]\label{4lemaadd2} For $ \alpha,\:y>0$, it holds that
\begin{equation}\aligned\nonumber
\mathcal{W}_{d,6}\geq \frac{1}{4{y}^3}e^{-\pi\alpha(y+\frac{1}{4y})}.
\endaligned\end{equation}
\end{lemma}
\begin{proof}
By \eqref{staro}, one has
\begin{equation}\aligned\label{star7}\nonumber
\mathcal{W}_{d,6}=\underset{p=2n-1,q=2m-1 }{\sum} \frac{{p}^2{q}^4}{16{y}^3} e^{-\pi\alpha(y{p}^2+\frac{q^2}{4y})}=\frac{1}{16{y}^3}\underset{n,m}{\sum}(2n-1)^2(2m-1)^4e^{-\pi\alpha[y(2n-1)^2+\frac{(2m-1)^2}{4y}]}.
\endaligned\end{equation}
Notice that all terms of the double sum $\mathcal{W}_{d,6}$ are non-negative. Then the term
\begin{equation}\aligned\nonumber
(n,m)\in \{(0,0),(0,1),(1,0),(1,1) \}\:\:\mathrm{contributes\:\: an \:\:amount}\:\frac{1}{16{y}^3}e^{-\pi\alpha(y+\frac{1}{4y})}\:\:\mathrm{respectively}.
\endaligned\end{equation}
\end{proof}

To estimate $\mathcal{W}_{e}(\alpha;y)$, we denote that
\begin{equation}\aligned\label{deno}
\mathcal{W}_{e,1}:=&\underset{p\equiv q\equiv0(mod2) }{\sum} y^2 {p}^6 e^{-\pi\alpha(y{p}^2+\frac{q^2}{4y})},\:\:\:\:\:
\mathcal{W}_{e,2}:=\underset{p\equiv q\equiv1(mod2) }{\sum}  y^2{p}^6 e^{-\pi\alpha(y{p}^2+\frac{q^2}{4y})},\\
\mathcal{W}_{e,3}:=&\underset{p\equiv q\equiv0(mod2) }{\sum}\frac{p^2q^4}{16{y}^2} e^{-\pi\alpha(y{p}^2+\frac{q^2}{4y})},\:\:
\mathcal{W}_{e,4}:=\underset{p\equiv q\equiv1(mod2) }{\sum}\frac{p^2q^4}{16{y}^2} e^{-\pi\alpha(y{p}^2+\frac{q^2}{4y})},\\
\mathcal{W}_{e,5}:=&\underset{p\equiv q\equiv0(mod2) }{\sum}\frac{{p}^4{q}^2}{2}e^{-\pi\alpha(y{p}^2+\frac{q^2}{4y})},\:\:
\mathcal{W}_{e,6}:=\underset{p\equiv q\equiv1(mod2) }{\sum} \frac{{p}^4{q}^2}{2} e^{-\pi\alpha(y{p}^2+\frac{q^2}{4y})}.\\
\endaligned\end{equation}

\begin{lemma}[An upper bound estimate of $\mathcal{W}_{e}(\alpha;y)$] \label{4lema} For $ \alpha,\:y>0$, it holds that
  \begin{equation}\aligned\nonumber
\mathcal{W}_{e}(\alpha;y)=\underset{n,m}{\sum}{n}^2(yn^2+\frac{(m+\frac{n}{2})^2}{y})^2e^{-\pi\alpha(yn^2+\frac{(m+\frac{n}{2})^2}{y})}
\leq 4y^2e^{-\pi\alpha(y+\frac{1}{4y})}\cdot(1+{\epsilon}_{c}),
\endaligned\end{equation}
where ${\epsilon}_{c}:=\underset{j=1}{\overset{3}{\sum}}\epsilon_{b,j}+\underset{j=1}{\overset{5}{\sum}}\epsilon_{c,j}.$
Here $\epsilon_{b,1},\epsilon_{b,2},\epsilon_{b,3}$ are defined in \eqref{eqb} and each $\epsilon_{c,j}$ is small \\
$(j=1,2,3,4,5)$ and expressed by
\begin{equation}\aligned\nonumber
&\epsilon_{c,1}:=32e^{-\pi\alpha(3y-\frac{1}{4y})}\cdot(1+\underset{n=2 }{\overset{\infty}{\sum}}n^{6}e^{-4\pi\alpha{y}({n}^2-1)})\cdot\vartheta_3(\frac{\alpha}{y}),\\
&\epsilon_{c,2}:=\frac{4}{y^4}e^{-3\pi\alpha(y+\frac{1}{4y})}(1+\underset{n=2}{\overset{\infty}{\sum}}n^{2}e^{-4{\pi}{\alpha}y(n^2-1)})
(1+\underset{m=2}{\overset{\infty}{\sum}}m^{4}e^{-\pi\alpha\frac{m^2-1}{y}}),\\
&\epsilon_{c,3}:=\frac{1}{16y^4}\big(1+\underset{n=2}{\overset{\infty}{\sum}}(2n-1)^2e^{-4{\pi}{\alpha}y(n-1)\cdot n} \big)\cdot\big(1+\underset{m=2}{\overset{\infty}{\sum}}(2m-1)^4e^{-\pi\alpha\frac{(m-1)\cdot m}{y}} \big),\\
&\epsilon_{c,4}:=\frac{32}{y^2}e^{-3\pi\alpha(y+\frac{1}{4y})}(1+\underset{n=2}{\overset{\infty}{\sum}}n^{4}e^{-4{\pi}{\alpha}y(n^2-1)})
(1+\underset{m=2}{\overset{\infty}{\sum}}m^{2}e^{-\pi\alpha\frac{m^2-1}{y}}),\\
&\epsilon_{c,5}:=\frac{1}{2y^2}\big(1+\underset{n=2}{\overset{\infty}{\sum}}(2n-1)^4e^{-4{\pi}{\alpha}y(n-1)\cdot n} \big)\cdot\big(1+\underset{m=2}{\overset{\infty}{\sum}}(2m-1)^2e^{-\pi\alpha\frac{(m-1)\cdot m}{y}} \big).\\
\endaligned\end{equation}
\end{lemma}
\begin{proof}
By \eqref{4lem3}, one has
\begin{equation}\aligned\nonumber
\mathcal{W}_{e}(\alpha;y)=\underset{p\equiv q(mod2) }{\sum} p^2(y p^2+\frac{q^2}{4y})^2e^{-\pi\alpha(y{p}^2+\frac{q^2}{4y})}=\underset{k=1}{\overset{6}{\sum}}\mathcal{W}_{e,k}(\alpha;y).
\endaligned\end{equation}
By \eqref{staro}, \eqref{star2} and \eqref{deno}, one gets
\begin{equation}\aligned\nonumber
&\mathcal{W}_{e,1}=y\mathcal{W}_{d,1}=4 y^2 e^{-\pi\alpha(y+\frac{1}{4y})}\epsilon_{c,1},\\
&\mathcal{W}_{e,2}=y\mathcal{W}_{d,2}=4 y^2 e^{-\pi\alpha(y+\frac{1}{4y})}\cdot(1+{\epsilon}_{b,1}+{\epsilon}_{b,2}+{\epsilon}_{b,3}).
\endaligned\end{equation}
Next, by \eqref{deno}, we shall calculate $\mathcal{W}_{e,3},\:\mathcal{W}_{e,4}$ sequentially.
\begin{equation}\aligned\label{eqe3}
\mathcal{W}_{e,3}&=\underset{p=2n,q=2m }{\sum}\frac{p^2q^4}{16{y}^2} e^{-\pi\alpha(y{p}^2+\frac{q^2}{4y})}=\frac{4}{y^2}\underset{n}{\sum}n^2e^{-4\pi\alpha yn^2}\underset{m}{\sum}m^4e^{-\pi\alpha\frac{m^2}{y}}\:\:\:\:\:\:\:\:\:\:\:\:\:\:\:\:\:\:\:\:\:\:\:\:\:\:\:\:\:\:\:\:\:\:\:\:\:\:\:\:\\
&=\frac{16}{y^2}e^{-4\pi\alpha(y+\frac{1}{4y})}(1+\underset{n=2}{\overset{\infty}{\sum}}n^{2}e^{-4{\pi}{\alpha}y(n^2-1)})
(1+\underset{m=2}{\overset{\infty}{\sum}}m^{4}e^{-\pi\alpha\frac{m^2-1}{y}})\\
&=4 y^2 e^{-\pi\alpha(y+\frac{1}{4y})}\cdot\epsilon_{c,2}.
\endaligned\end{equation}
\begin{equation}\aligned\label{eqe4}
\mathcal{W}_{e,4}&=\underset{p=2n-1,q=2m-1 }{\sum}\frac{p^2q^4}{16{y}^2} e^{-\pi\alpha(y{p}^2+\frac{q^2}{4y})}\\
&=\frac{1}{16y^2}\underset{n}{\sum}(2n-1)^2e^{-\pi\alpha y(2n-1)^2}\underset{m}{\sum}(2m-1)^4e^{-\pi\alpha\frac{(2m-1)^2}{4y}}\\
&=\frac{1}{4y^2}\underset{n=1}{\overset{\infty}{\sum}}(2n-1)^2e^{-\pi\alpha y(2n-1)^2}\underset{m=1}{\overset{\infty}{\sum}}(2m-1)^4e^{-\pi\alpha\frac{(2m-1)^2}{4y}}\\
&=\frac{1}{4y^2}e^{-\pi\alpha(y+\frac{1}{4y})}\big(1+\underset{n=2}{\overset{\infty}{\sum}}(2n-1)^2e^{-4{\pi}{\alpha}y(n-1)\cdot n} \big)\cdot\big(1+\underset{m=2}{\overset{\infty}{\sum}}(2m-1)^4e^{-\pi\alpha\frac{(m-1)\cdot m}{y}} \big)\\
&=4 y^2 e^{-\pi\alpha(y+\frac{1}{4y})}\cdot\epsilon_{c,3}.
\endaligned\end{equation}
Similar to the calculations in \eqref{eqe3}-\eqref{eqe4}, one has
\begin{equation}\aligned\nonumber
\mathcal{W}_{e,5}=4 y^2 e^{-\pi\alpha(y+\frac{1}{4y})}\cdot\epsilon_{c,4},\:\:\mathcal{W}_{e,6}=4 y^2 e^{-\pi\alpha(y+\frac{1}{4y})}\cdot\epsilon_{c,5}.
\endaligned\end{equation}
\end{proof}

\section{Proof of Theorem \ref{Th2} and Corollary \ref{cor1}}
Proof of Theorem \ref{Th2}. Theorems \ref{2Th1} and \ref{4Th1} proves Theorem \ref{Th2}.

Proof of Corollary \ref{cor1}. Applying the fundamental Theorem of calculus on the parameters $\alpha, \beta$, we have

\begin{equation}\aligned\nonumber
\sum_{\mathbb{P}\in \Lambda} {\left| \mathbb{P} \right|}^{2}\Big(e^{-\pi\alpha {\left| \mathbb{P} \right|}^{2}}  -e^{-\pi\beta {\left| \mathbb{P} \right|}^{2}} \Big)
=\pi\int_\alpha^\beta\sum_{\mathbb{P}\in \Lambda} {\left| \mathbb{P} \right|}^{4}e^{-\pi\gamma {\left| \mathbb{P} \right|}^{2}}d\gamma.
\endaligned\end{equation}
Then Corollary \ref{cor1} follows by Theorem \ref{Th2}.

\bigskip
\noindent
{\bf Acknowledgements.}
 The research of S. Luo is partially supported by NSFC(12001253), double thousands plan of Jiangxi(jxsq2019101048) and Jiangxi Jieqing fund(20242BAB23001).
\vskip0.1in
{\bf Statements and Declarations: there is no conflict of interest.}

{\bf Data availability: the manuscript has no associated data.}
\bigskip


\end{document}